\documentclass[a4paper,11pt]{scrartcl}
\addtokomafont{disposition}{\rmfamily}

\usepackage{amsthm}
\usepackage{amsmath}
\usepackage{amssymb}
\usepackage{mathrsfs} 
\usepackage{graphicx}
\usepackage[round]{natbib}

\usepackage[dvipsnames]{xcolor}

\usepackage[colorlinks=true,linkcolor=NavyBlue, citecolor=NavyBlue, urlcolor=NavyBlue, linktocpage=true]{hyperref} 
\usepackage{bbm}
\usepackage{textcomp}
\usepackage{enumerate}
\usepackage{mathtools}

\usepackage{color}
\usepackage{framed}
\usepackage{comment}
\definecolor{shadecolor}{gray}{0.9}

\usepackage{dsfont} 
\usepackage{caption} 
\usepackage{subcaption} 
\usepackage{graphicx}
\usepackage{pdfpages}
\usepackage{url}
\usepackage[english]{babel}
\specialcomment{extra}{\begin{shaded}}{\end{shaded}}

\theoremstyle{plain}  
\newtheorem{thm}{Theorem}[section] 
\newtheorem{lem}[thm]{Lemma} 
\newtheorem{prop}[thm]{Proposition} 
\newtheorem{cor}[thm]{Corollary} 

\theoremstyle{definition} 
\newtheorem{defn}[thm]{Definition}
 
\newtheorem{exmp}[thm]{Example}
\newtheorem{rem}[thm]{Remark}

\newtheoremstyle{assumption}
{3pt}
{3pt}
{}
{}
{\bf}
{.}
{.5em}
{\thmname{#1} (\thmnote{#3}\thmnumber{#2})}

\theoremstyle{assumption}

\theoremstyle{remark} 

\newcommand{\diff}{\mathrm{d}}
\newcommand{\dint}{\,\mathrm{d}}
\newcommand{\E}{\mathbb{E}}

\newcommand{\ES}{\operatorname{ES}}

\newcommand{\Var}{\operatorname{Var}}

\newcommand{\eps}{\varepsilon}

\newcommand{\supp}{\operatorname{supp}}

\newcommand{\F}{\mathcal{F}}
\newcommand{\R}{\mathbb{R}}
\newcommand{\A}{\mathsf{A}}
\renewcommand{\O}{\mathsf{O}}
\newcommand{\one}{\mathds{1}}
\newcommand{\interior}{\operatorname{int}}

\newcommand{\rank}{\operatorname{rank}}

\newcommand{\ph}{\varphi}

\newcommand{\VaR}{\operatorname{VaR}}
\DeclareMathOperator*{\argmin}{arg\,min}
\newcommand{\sgn}{\operatorname{sgn}}

\newcommand{\N}{\mathbb N}
\renewcommand{\P}{\mathbb P}
\renewcommand{\L}{\mathcal L}

\renewcommand{\a}{\alpha}

\renewcommand{\rm}{\normalfont \rmfamily}
\renewcommand{\bf}{\normalfont \bfseries}

\def\be{\begin{equation} \label}
\def\ee{\end{equation}}

\numberwithin{equation}{section} 

\begin{document}

\title{Order-Sensitivity and Equivariance of Scoring Functions}
\author{Tobias Fissler\thanks{Imperial College London, Department of Mathematics, Section Statistics, 180 Queen's Gate,
London SW7 2AZ, United Kingdom, 
e-mail: \texttt{t.fissler@imperial.ac.uk}} \and Johanna F.~Ziegel\thanks{University of Bern, Department of Mathematics and Statistics, Institute of Mathematical Statistics and Actuarial Science, Sidlerstrasse 5, 3012 Bern, Switzerland, 
e-mail: \texttt{johanna.ziegel@stat.unibe.ch}}}
\maketitle

\begin{abstract}
\textbf{Abstract.}
The relative performance of competing point forecasts is usually measured in terms of loss or scoring functions. It is widely accepted that these scoring function should be strictly consistent in the sense that the expected score is minimized by the correctly specified forecast for a certain statistical functional such as the mean, median, or a certain risk measure. Thus, strict consistency opens the way to meaningful forecast comparison, but is also important in regression and M-estimation. 
Usually strictly consistent scoring functions for an elicitable functional are not unique. To give guidance on the choice of a scoring function, this paper introduces two additional quality criteria. Order-sensitivity opens the possibility to compare two deliberately misspecified forecasts given that the forecasts are ordered in a certain sense. On the other hand, equivariant scoring functions obey similar equivariance properties as the functional at hand -- such as translation invariance or positive homogeneity.
In our study, we consider scoring functions for popular functionals, putting special emphasis on vector-valued functionals, e.g.\ the pair (mean, variance) or (Value at Risk, Expected Shortfall).
\end{abstract}


\section{Introduction}

From the cradle to the grave, human life is full of decisions. Due to the inherent nature of time, decisions have to be made today, but at the same time, they are supposed to account for unknown and uncertain future events. 
However, since these future events cannot be \emph{known} today, the best thing to do is to base the decisions on \emph{predictions} for these unknown and uncertain events. The call for and the usage of predictions for future events is literally ubiquitous and even dates back to ancient times. 
In those days, dreams, divination, and revelation were considered as respected sources for forecasts, with the most prominent example being the Delphic Oracle which was not only consulted for decisions of private life, but also for strategic political decisions concerning peace and war.
With the development of natural sciences, mathematics, and in particular statistics and probability theory, the ancient metaphysical art of making qualitative forecasts turned into a sophisticated discipline of science adopting a quantitative perspective. Subfields such as meteorology, mathematical finance, or even futurology evolved.

Acknowledging that forecasts are inherently uncertain, two main questions arise:
\begin{enumerate}[(i)]
\item
How good is a forecast in \emph{absolute} terms?
\item
How good is a forecast in \emph{relative} terms?
\end{enumerate}
While question (i) deals with \emph{forecast validation}, this paper focuses on some aspects of question (ii) which is concerned with \emph{forecast selection}, \emph{forecast comparison}, or \emph{forecast ranking}. Specifically, we present results on order-sensitivity and equivariance of consistent scoring functions for elicitable functionals. These results may provide guidance for choosing a specific scoring function for forecast comparison within the large class of all consistent scoring functions for an elicitable functional of interest.

We adopt the general decision-theoretic framework following \cite{Gneiting2011}; cf.~\citet{Savage1971, Osband1985, Lambert2008}. For some number $n\ge1$, one has
\begin{enumerate}[(a)]
\item
observed \emph{ex post} realizations $y_1, \ldots, y_n$ of a time series $(Y_t)_{t\in \N}$, taking values in an \emph{observation domain} $\O$ with a $\sigma$-algebra $\mathcal O$;
\item
a family $\F$ of probability distributions on $(\O, \mathcal O)$, containing the (conditional) distributions of $Y_t$;
\item
\emph{ex ante} forecasts $x_1^{(i)}, \ldots, x_n^{(i)}$, $i\in\{1, \ldots, m\}$ of $m\ge1$ competing experts\,/\,forecasters taking values in an \emph{action domain} $\A \subseteq \mathbb{R}^k$ for some $k \ge 1$;
\item
a \emph{scoring (or loss) function} $S\colon\A\times \O\to\R$. The scoring function is assumed to be negatively oriented, that is, if a forecaster reports the quantity $x\in\A$ and $y\in\O$ materializes, she is assigned the \emph{penalty} $S(x,y) \in\R$.
\end{enumerate}
The observations $y_t$ can be real-valued (GDP growth for one year, maximal temperature of one day), vector-valued (wind-speed, weight and height of persons), functional-valued (path of the exchange rate Euro--Swiss franc over one day), or also set-valued (area of rain on a given day, area affected by a flood). In this article, we focus on \emph{point forecasts} that may be vector-valued, which is why we assume $\mathsf{A} \subseteq \mathbb{R}^k$ for some $k \ge 1$ and we equip the Borel set $\mathsf{A}$ with the Borel $\sigma$-algebra. One is typically interested in a certain statistical property of the underlying (conditional) distribution $F_t$ of $Y_t$. We assume that this property can be expressed in terms of a \emph{functional} $T\colon\F\to\A$ such as the mean, a certain quantile, or a risk measure. Examples of vector-valued functionals are the covariance matrix of a multivariate observation or a vector of quantiles at different levels. Common examples for scoring functions are the absolute loss $S(x,y) = |x-y|$, the squared loss $S(x,y) = (x-y)^2$ (for $\A=\O = \R$), or the absolute percentage loss $S(x,y) = |(x-y)/y|$ (for $\A = \O = (0,\infty)$). 

Forecast comparison is done in terms of \emph{realized scores}
\begin{equation}\label{eq:realscore}
\bar{\mathbf{S}}_n^{(i)} = \frac{1}{n} \sum_{t=1}^n S(x_t^{(i)}, y_t), \qquad i\in\{1, \ldots, m\}.
\end{equation}
That is, a forecaster is deemed to be the better the lower her realized score is. However, there is the following caveat: The forecast ranking in terms of realized scores not only depends on the forecasts and the realizations (as it should definitely be the case), but also on the choice of the scoring function. In order to avoid impure possibilities of manipulating the forecast ranking \emph{ex post} with the data at hand, it is necessary to specify a certain scoring function before the inspection of the data. \emph{A fortiori}, for the sake of transparency and in order to encourage truthful forecasts, one ought to disclose the choice of the scoring function to the competing forecasters \emph{ex ante}. But still, the optimal choice of the scoring function remains an open problem. One can think of two situations:
\begin{enumerate}[(i)]
\item
A decision-maker might be aware of his actual economic costs of utilizing misspecified forecasts. In this case, the scoring function should reflect these economic costs.
\item
The actual economic costs might be unclear and the scoring function might be just a tool for forecast ranking. However, the directive is given in terms of the functional $T\colon \F\to\A$ one is interested in.
\end{enumerate}

For situation (i) described above, one should use the readily economically interpretable cost or scoring function. Therefore, the only concern is situation (ii).    In this paper, we consider predictions in a one-period setting, thus, dropping the index $t$. This is justified by our objectives to understand the properties of scoring functions $S$ which do not change over time and is common in the literature \citep{MurphyDaan1985, DieboldMariano1995,Lambert2008,Gneiting2011}.

Assuming the forecasters are \emph{homines oeconomici} and adopting the rationale of expected utility maximization, given a concrete scoring function $S$, the most sensible action consists in minimizing the expected score $\E_F S(x,Y)$ with respect to the forecast $x$, where $Y$ follows the distribution $F$, thus issuing the Bayes act $\argmin_{x\in\A} \E_F S(x,Y)$. Hence, a scoring function should be incentive compatible in that it encourages truthful and honest forecasts. In line with \cite{MurphyDaan1985} and \cite{Gneiting2011}, we make the following definition.
\begin{defn}[Consistency and elicitability]\label{def:consistency}
A \textit{scoring function} is a map $S\colon \A\times \O \to \R$ that is $\F$-integrable.\footnote{We say that a function $a\colon\O\to\R$ is $\F$-integrable if it is $F$-integrable for each $F\in\F$. A function $g\colon \A \times \O \to \R$ is $\F$-integrable if $g(x,\cdot)$ is $\F$-integrable for each $x\in \A$.}
It is \textit{$\F$-consistent} for a functional $T\colon \F \to \A$ if
\begin{equation}\label{eq:1}
\bar{S}(T(F),F)\le \bar{S}(x,F)
\end{equation}
for all $F\in \F$ and for all $x\in \A$, where $\bar{S}(x,F):=\E_F S(x,Y)$. It is \textit{strictly $\F$-consistent} for $T$ if it is $\F$-consistent for $T$ and if equality in \eqref{eq:1} implies $x = T(F)$.
A functional $T\colon \F\to \A$ is called \textit{elicitable}, if there exists a strictly $\F$-consistent scoring function for $T$.
\end{defn}
Clearly, elicitability and consistent scoring functions are naturally linked also to estimation problems, in particular, M-estimation \citep{Huber1964, HuberRonchetti2009} and regression with prominent examples being ordinary least squares, quantile, or expectile regression \citep{Koenker2005, NeweyPowell1987}. 

The necessity of utilizing strictly consistent scoring functions for meaningful forecast comparison is impressively demonstrated in terms of a simulation study in \cite{Gneiting2011}. However, for a given functional $T\colon\F \to \A$, there is typically a whole class of strictly consistent scoring functions for it, such as all Bregman functions in case of the mean \citep{Savage1971}; further examples are given below. \citet{Patton2016} shows that the forecast ranking based on \eqref{eq:realscore} may depend on the choice of the strictly consistent scoring function for $T$ in finite samples, and even at the population level if we compare two imperfect forecasts with each other. 

Therefore, we naturally have a threefold \emph{elicitation problem}:
\begin{enumerate}[(i)]
\item
Is $T$ elicitable?
\item
What is the class of strictly $\F$-consistent scoring functions for $T$?
\item
What are \emph{distinguished} strictly $\F$-consistent scoring functions for $T$? 
\end{enumerate}
Even though the denomination and the synopsis of the described problems under the term `elicitation problem' are novel, there is a rich strand of literature in mathematical statistics and economics concerned with the threefold elicitation problem. Foremost, one should mention the pioneering work of \cite{Osband1985}, establishing a necessary condition for elicitability in terms of convex level sets of the functional, and a necessary representation of strictly consistent scoring functions, known as \emph{Osband's principle} \citep{Gneiting2011}. 
Whereas the necessity of convex level sets holds in broad generality, \cite{Lambert2013} could specify sufficient conditions for elicitability for functionals taking values in a finite set, and \cite{SteinwartPasinETAL2014} showed sufficiency of convex level sets for real-valued functionals satisfying certain regularity conditions. Moments, ratios of moments, quantiles, and expectiles are in general elicitable, whereas other important functionals such as variance, Expected Shortfall or the mode functional are not \citep{Savage1971, Osband1985, Weber2006, Gneiting2011, Heinrich2014}. 

Concerning subproblem (ii) of the elicitation problem, \cite{Savage1971}, \cite{ReichelsteinOsband1984}, \cite{Saerens2000}, and \cite{BanerjeeGuoETAL2005} gave characterizations for strictly consistent scoring functions for the mean functional of a one-dimensional random variable in terms of Bregman functions. Strictly consistent scoring functions for quantiles have been characterized by \cite{Thomson1979} and \cite{Saerens2000}. \cite{Gneiting2011} provides a characterization of the class of strictly consistent scoring functions for expectiles. The case of vector-valued functionals apart from means of random vectors has been treated substantially less than the one-dimensional case \citep{Osband1985, BanerjeeGuoETAL2005, Lambert2008, FrongilloKash2014b, FrongilloKash2015, FisslerZiegel2016}.

The strict consistency of $S$ only justifies a comparison of two competing forecasts if one of them reports the true functional value. If both of them are misspecified, it is \textit{per se} not possible to draw a conclusion which forecast is `closer' to the true functional value by comparing the realized scores. To this end, some notions of order-sensitivity are desirable. According to \cite{Lambert2013} we say that a scoring function $S$ is \emph{$\F$-order-sensitive} for a one-dimensional functional $T\colon \F\to\A\subseteq \R$ if for any $F\in\F$ and any $x,z\in\A$ such that either $z\le x \le T(F)$ or $z\ge x\ge T(F)$, then $\bar S(x,F)\le \bar S(z,F)$. This means, if a forecast lies between the true functional value and some other forecast, then issuing the forecast in-between should yield a smaller expected score than issuing the forecast further away. In particular, order-sensitivity implies consistency. \emph{Vice versa}, under weak regularity conditions on the functional, strict consistency also implies order-sensitivity if the functional is real-valued; see \citet[Proposition 3]{Nau1985}, \citet[Proposition 2]{Lambert2013}, \citet[Proposition 3.4]{BelliniBignozzi2015}. 

This article is dedicated to a thorough investigation of order-sensitive scoring functions for vector-valued functionals, thus contributing to a discussion of subproblem (iii) of the elicitation problem. Furthermore, we investigate to which extent invariance or equivariance properties of elicitable functionals are reflected in their respective consistent scoring functions. 

\cite{Lambert2008} introduced a notion of componentwise order-sensitivity for the case of $\A\subseteq\R^k$. \cite{Friedman1983} and \cite{Nau1985} considered similar questions in the setting of probabilistic forecasts, coining the term of effectiveness of scoring rules which can be described as order-sensitivity in terms of a metric. In Section \ref{sec:order-sensitivity}, we consider three notions of order-sensitivity in the higher-dimensional setting: metrical order-sensitivity, componentwise order-sensitivity, and order-sensitivity on line segments. We discuss their connections and give conditions when such scoring functions exist and of what form they are for the most relevant functionals, such as vectors of quantiles, expectiles, ratios of expectations, the pair of mean and variance, and the pair consisting of Value at Risk and Expected Shortfall, two important risk measures in banking and insurance.


Complementing our results on order-sensitivity, in Section \ref{sec:analytic}, we consider the analytic properties of the expected score $x\mapsto \bar S(x,F)$, $x\in\A \subseteq \mathbb{R}^k$, for some scoring function $S$ and some distribution $F\in\F$. The (strict) consistency of $S$ for some functional $T$ is equivalent the expected score having a (unique) global minimum at $x=T(F)$. Order-sensitivity ensures monotonicity properties of the expected score. 
As a technical result, we show that under weak regularity assumptions on $T$, the expected score of a strictly consistent scoring function has a unique local minimum -- which, of course, coincides with the global minimum at $x=T(F)$. Accompanied with a result on self-calibration, a continuity property of the inverse of the expected score, which ensures that the minimum of the expected score is well-separated in the sense of \cite{Vaart1998}, these two findings may be of interest on their own right in the context of M-estimation.

In Section 4, we consider functionals that have an invariance or equivariance property such as translation invariance or homogeneity. It is a natural question whether a functional $T$ that is, for example, translation equivariant has a consistent scoring function that respects this property in the sense that if we evaluate forecast performance of translated predictions and observations, the ranking of predictive performance remains the same as that of the original data. In parametric estimation problems, such a scoring functions may allow to translate the data without affecting the estimated parameter values. For one-dimensional functionals, invariance of the scoring function often determines it uniquely up to equivalence while this is not necessarily the case for higher-dimensional functionals (Proposition \ref{prop:translation invariance} and Corollary \ref{cor:EVarHOM}).

\section{Analytic properties of expected scores}\label{sec:analytic}

\subsection{Monotonicity}

\begin{defn}[Mixture-continuity]\label{defn:Mixture-Continuity}
Let $\F$ be convex. A functional $T\colon \F\to\A\subseteq \R^k$ is called \emph{mixture-continuous} if for all $F,G\in \F$ the map
\[
[0,1]\to\R, \quad \lambda \mapsto T((1-\lambda)F + \lambda G)
\]
is continuous.
\end{defn}

It is appealing that one does not have to specify a topology on $\F$ to define mixture-continuity because it suffices to work with the induced Euclidean topology on $[0,1]$ and on $\A\subseteq \R^k$. 

It turns out that mixture-continuity of a functional is strong enough to imply order-sensitivity in the case of one-dimensional functionals (see \citet[Proposition 3]{Nau1985}, \citet[Proposition 2]{Lambert2013}, \citet[Proposition 3.4]{BelliniBignozzi2015}), and desirable monotonicity properties of the expected scores also in higher dimensions (Propositions \ref{prop:decr paths} and \ref{prop:no local minimum}). At the same time, numerous functionals of applied relevance are mixture-continuous, and we start by giving examples and a sufficient condition (Proposition \ref{prop:mixture-continuity}).

It is straight forward to see that the ratio of expectations is mixture-continuous. Moreover, by the implicit function theorem, one can verify the mixture-continuity of quantiles and expectiles directly under appropriate regularity conditions (e.g.,~in the case of quantiles, all distributions in $\F$ should be $C^1$ with non-vanishing derivatives). Generalizing \citet[Proposition 3.4c]{BelliniBignozzi2015}, we give a sufficient criterion for mixture-continuity in the next proposition. Our version is not restricted to distributions with compact support (however, the image of the functional must be bounded), and we formulate the result for $k$-dimensional functionals.

\begin{prop}\label{prop:mixture-continuity}
Let $T\colon \F\to \R^k$ be an elicitable functional with a strictly $\F$-consistent scoring function $S\colon \R^k\times \O \to\R$ such that $\bar S(\cdot, F)$ is continuous for all $F\in \F$. Then $T$ is mixture-continuous on any $\F_0\subseteq \F$ such that $\F_0$ is convex and the image $T(\F_0)$ is bounded.
\end{prop}
\begin{proof}
Let $\F_0\subseteq \F$ be convex such that $T(\F_0)\subseteq [-C,C]^k$ for some $C>0$. Let $F,G\in\F_0$. Define $h_{F,G} \colon [-C,C]^k \times [0,1]\to\R$ via
\[
h_{F,G}(x,\lambda) = \bar S(x,(1-\lambda)F + \lambda G) = (1-\lambda)\bar S(x,F) + \lambda \bar S(x,G).
\]
Then $h_{F,G}$ is jointly continuous, and due to the strict consistency 
\[
T((1-\lambda) F + \lambda G) = \argmin_{x\in[-C, C]^k} h_{F,G}(x,\lambda). 
\]
By virtue of the Berge Maximum Theorem \citep[Theorem 17.31 and Lemma 17.6]{AliprantisBorder2006}, the function $\lambda\mapsto \argmin_{x\in[-C, C]^k} h_{F,G}(x,\lambda)$
is continuous.
\end{proof}

Similarly to the original proof of \cite{BelliniBignozzi2015}, a sufficient criterion for the continuity of $\bar S(\cdot, F)$ for any $F\in\F$ is that for all $y\in\O$, the score $S(x,y)$ is quasi-convex and continuous in $x$.\footnote{We remark that for $\A\subseteq \R$, if a scoring function $S$ is strictly $\F_p$-consistent for some functional $T\colon\F_p\to\A$ where $\F_p = \{\delta_y\colon y\in\O\}$ consists of all point measures on $\O$, then the quasi-convexity of $x\mapsto S(x,y)$ for all $y\in\O$ is equivalent to the $\F_p$-order-sensitivity of $S$ for $T$.}

Recall that, under appropriate regularity conditions on $\F$, the asymmetric piecewise linear loss $S_\alpha(x,y) = (\one\{y\le x\} - \alpha)(x-y)$ and the asymmetric piecewise quadratic loss $S_\tau(x,y) = |\one\{y\le x\} - \tau|(x-y)^2$ are strictly consistent scoring functions for the $\alpha$-quantile and the $\tau$-expectile, respectively, and both, $S_\alpha$ as well as $S_\tau$, are continuous in their first argument and convex. Hence, Proposition \ref{prop:mixture-continuity} yields that both quantiles and expectiles are mixture-continuous.

\cite{SteinwartPasinETAL2014} used Osband's principle \citep{Osband1985} and the assumption of continuity of $T$ with respect to the total variation distance to show order-sensitivity. \cite{BelliniBignozzi2015} showed that the weak continuity of a functional $T$ implies its mixture-continuity. Consequently, one can also derive the order-sensitivity in the framework of \cite{SteinwartPasinETAL2014} directly using only mixture-continuity.

\cite{Lambert2013} showed that it is a harder requirement to have order-sensitivity if $T(\F)$ is discrete. Then both approaches, invoking Osband's principle or using mixture-continuity, do not work because the interior of the image of $T$ is empty. Moreover, mixture-continuity implies that the functional is constant (such that only trivial cases can be considered). Furthermore, it is proven in \cite{Lambert2013} that for a functional $T$ with a discrete image, all strictly consistent scoring functions are order-sensitive if and only if there is one order-sensitive scoring function for $T$.
In particular, there are functionals admitting strictly consistent scoring functions that are not order-sensitive, one such example being the mode functional.\footnote{Note that due to Proposition 1 in \cite{Heinrich2014}, the mode functional is elicitable relative to the class of probability measure $\F$ containing unimodal discrete measures. Moreover, interpreting the mode functional as a set-valued functional, it is elicitable in the sense of \citet[Definition 2]{Gneiting2011}. A strictly $\F$-consistent scoring function is given by $S(x,y) = \one\{x\neq y\}$. The main result of \cite{Heinrich2014} is that the mode functional is \emph{not} elicitable relative to the class $\F$ of unimodal probability measures with Lebesgue densities.}


Let us turn attention to vector-valued functionals now. To understand the monotonicity properties of the expected score of a mixture-continuous elicitable functional $T\colon\F \to \A\subseteq\R^k$, it is useful to consider paths $\gamma\colon [0,1]\to\A\subseteq \R^k$, $\gamma(\lambda) = T(\lambda F + (1-\lambda)G)$ for $F,G\in\F$. If $T$ is elicitable, a classical result asserts that $T$ necessarily has convex level sets \cite[Theorem 6]{Gneiting2011}. This implies that the level sets of $\gamma$ can only be closed intervals including the case of singletons and the empty set. This rules out loops and some other possible pathologies of $\gamma$. Furthermore, under the assumption that $T$ is \emph{identifiable} as defined below, one can even show that the path $\gamma$ is either injective or constant. 

\begin{defn}[Identifiability]
Let $\A\subseteq\R^k$.
An $\F$-integrable function $V\colon \A\times \O \to \R^k$ is said to be an \textit{$\F$-identification function} for a functional $T\colon \F \to \A\subseteq \R^k$ if
\[
\bar V(T(F),F)=0
\]
for all $F\in\F$. Furthermore, $V$ is a \textit{strict $\F$-identification function} for $T$ if $\bar V(x,F) = 0$ implies $x = T(F)$ for all $F\in \F$ and for all $x\in \A$. A functional $T\colon \F\to \A\subseteq\R^k$ is said to be \textit{identifiable}, if there exists a strict $\F$-identification function for $T$.
\end{defn}
In line with \citet[Section 2.4]{Gneiting2011}, one can often obtain an identification function as the gradient of a sufficiently smooth scoring function. However, the converse intuition is not so clear -- at least in the higher dimensional setting $k>1$: Not all strict identification functions can be integrated to a strictly consistent scoring function. They have to satisfy the usual integrability conditions \cite[p.\ 185]{Koenigsberger2004}; see also \citet[Corollary 3.3]{FisslerZiegel2016} and the discussion thereafter.

\begin{lem}\label{lem:injective paths}
Let $\F$ be convex and $T\colon\F\to\A\subseteq \R^k$ be identifiable with a strict $\F$-identification function $V\colon \A\times\O\to\R^k$. Then for any $F,G\in\F$, the path $\gamma\colon [0,1]\to\A$, $\gamma(\lambda) = T(\lambda F + (1-\lambda)G)$, is either constant or injective.
\end{lem}

\begin{proof}
Let $F,G\in\F$ such that $t=T(F) = T(G)$. For any $\lambda\in[0,1]$, one has $\bar V(t,\lambda F + (1-\lambda)G) = \lambda \bar V(t,F) + (1-\lambda)\bar V(t,G) = 0$. Since $V$ is a strict $\F$-identification function for $T$, $t = \gamma(\lambda)$ for all $\lambda \in[0,1]$. 

Now let $T(F) \neq T(G)$ and let $0\le \lambda<\lambda'\le 1$. Since $V$ is a strict $\F$-identification function, $\bar V(T(F),G)\neq0$ (and symmetrically $\bar V(T(G),F)\neq0$.) Assume that $\gamma(\lambda) = \gamma(\lambda')$. Define $H_\lambda = \lambda F + (1-\lambda) G$, $H_{\lambda'} = \lambda' F + (1-\lambda') G$. 
There are $\mu, \mu'\in\R$ such that $F = \mu H_\lambda + (1-\mu)H_{\lambda'}$ and $G = \mu' H_\lambda + (1-\mu')H_{\lambda'}$. Hence,
\[
\bar V(\gamma(\lambda),F) = \mu\bar V(\gamma(\lambda),H_{\lambda}) + (1-\mu)\bar V(\gamma(\lambda),H_{\lambda'}) = 0,
\]
and similarly
\(
\bar V(\gamma(\lambda),G) = 0.
\)
Consequently, $T(F) = \gamma(\lambda) = T(G)$, which is a contradiction to the assumption that $T(F) \neq T(G)$. This implies that $\gamma(\lambda) \neq \gamma(\lambda')$.
\end{proof}

\begin{prop}\label{prop:decr paths}
Let $\F$ be convex and $T\colon\F\to\A\subseteq \R^k$ be mixture-continuous and surjective. Let $S\colon \A\times \O\to\R$ be strictly $\F$-consistent for $T$. Then for each $F\in \F$, $t=T(F)$ and each $x\in\A$, $x\neq t$ there is a continuous path $\gamma\colon [0,1]\to\A$ such that $\gamma(0) = x$, $\gamma(1)=t$, and the function $[0,1]\ni\lambda \mapsto \bar S(\gamma(\lambda),F)$ is decreasing. Additionally, for $0\le \lambda<\lambda'\le 1$ such that $\gamma (\lambda )\neq \gamma(\lambda')$ it holds that $\bar S(\gamma(\lambda),F)>\bar S(\gamma(\lambda'),F)$.
\end{prop}
\begin{proof}
Let $F\in\F$, $t=T(F)$ and $x\neq t$. Then there is some $G\in\F$ with $x=T(G)$. Define $\gamma\colon [0,1]\to \A, \quad \lambda \mapsto T(\lambda F +(1-\lambda) G)$. Clearly, $\gamma(0)=x$ and $\gamma(1)=t$. Due to the mixture-continuity of $T$, the path $\gamma$ is also continuous. The rest follows along the lines of the proof of \citet[Proposition 3]{Nau1985}.
Let $0\le \lambda<\lambda'\le 1$. If $\gamma(\lambda) = \gamma(\lambda')$, there is nothing to show. So assume that $\gamma(\lambda) \neq \gamma(\lambda')$.
Define $H_\lambda = \lambda F + (1-\lambda)G$, and $H_{\lambda'}$ analogously. Then, for $\mu := (\lambda'- \lambda)/(1-\lambda)\in(0,1]$, it holds that $H_{\lambda'}= \mu F +(1-\mu) H_\lambda$. The strict consistency of $S$ implies that 
\begin{multline*}
\mu \bar S(\gamma(\lambda'), F) + (1-\mu)\bar S(\gamma(\lambda'),H_\lambda) 
= \bar S(\gamma(\lambda'),H_{\lambda'}) \\
< \bar S(\gamma(\lambda),H_{\lambda'})
=\mu \bar S(\gamma(\lambda), F) + (1-\mu)\bar S(\gamma(\lambda),H_\lambda) \,,
\end{multline*}
which is equivalent to
\[
\frac{1-\mu}{\mu} \big( \bar S(\gamma(\lambda'),H_\lambda) -\bar S(\gamma(\lambda),H_\lambda) \big)
<\bar S(\gamma(\lambda), F) - \bar S(\gamma(\lambda'), F)\,.
\]
By strict consistency of $S$, the left-hand side is non-negative yielding the assertion.
\end{proof}
\begin{rem}\begin{enumerate}
\item[(i)]
Proposition \ref{prop:decr paths} remains valid if $S$ is only $\F$-consistent. Then, we merely have that the function $[0,1]\ni\lambda \mapsto \bar S(\gamma(\lambda),F)$ is decreasing, so the last inequality in Proposition \ref{prop:decr paths} is not necessarily strict. 
\item[(ii)]
If one assumes in Proposition \ref{prop:decr paths} that $T$ is also identifiable, one can use the injectivity of $\gamma$ implied by Lemma \ref{lem:injective paths} to see that the function $[0,1]\ni \lambda \mapsto \bar S(\gamma(\lambda), F)$ is strictly decreasing.
\end{enumerate}
\end{rem}

Under certain (weak) regularity conditions, the expected scores of a strictly consistent scoring function has no other local minimum apart from the global one at $x=T(F)$.

\begin{prop}\label{prop:no local minimum}
Let $\F$ be convex and $T\colon\F\to\A\subseteq \R^k$ be mixture-continuous and surjective. If $S\colon \A\times \O\to\R$ is strictly $\F$-consistent for $T$, then for all $F\in\F$ the expected score $\bar S(\cdot, F)\colon \A\to\R$ has only one local minimum which is at $x=T(F)$. 
\end{prop}

\begin{proof}
Let $F\in\F$ with $t=T(F)$. Due to the strict $\F$-consistency of $S$, the expected score $\bar S(\cdot, F)$ has a local minimum at $t$. Assume there is another local minimum at some $x\neq t$. Then there is a distribution $G\in\F$ with $x=T(G)$. Consider the path $\gamma\colon [0,1]\to \A, \quad \lambda \mapsto T(\lambda F +(1-\lambda) G)$. Due to Proposition \ref{prop:decr paths} the function $\lambda \mapsto \bar S(\gamma (\lambda), F)$ is decreasing and strictly decreasing when we move on the image of the path from $x$ to $t$. Hence $\bar S(\cdot,F)$ cannot have a local minimum at $x = \gamma(0)$.
\end{proof}

\subsection{Self-calibration}\label{subsec:self-calibration}

With Proposition \ref{prop:decr paths} it is possible to prove that, under mild regularity conditions, strictly consistent scoring functions are \emph{self-calibrated} which turns out to be useful in the context of M-estimation.

\begin{defn}[Self-calibration]\label{defn:self-calibration}
A scoring function $S\colon\A\times \O\to\R$ is called $\F$-self-calibrated for a functional $T\colon \F\to\A\subseteq\R^k$ with respect to a norm\footnote{It is straight forward to use a metric instead of a norm on $\A$ but in this article we only consider $\A\subseteq\R^k$, so we did not see any benefit in considering this more general case. See also the discussion before Definition \ref{defn:Metrical order-sensitivity}.} $\|\cdot\|$ on $\A$ if for all $\eps>0$ and for all $F\in\F$ there is a $\delta= \delta(\eps,F)>0$ such that for all $x\in \A$ and $t=T(F)$
\[
\bar S(x,F) - \bar S(t,F) < \delta \quad\implies \quad \|t-x\|<\eps.
\]
\end{defn}

The notion of self-calibration was introduced by \cite{Steinwart2007} in the context of machine learning. In a preprint version of \cite{SteinwartPasinETAL2014},\footnote{Available at
\url{http://users.cecs.anu.edu.au/~williams/papers/P196.pdf}} the authors translate this concept to the setting of scoring functions as follows (using our notation):
\begin{quote}
\small{``For self-calibrated $S$, every $\delta$-approximate minimizer of $\bar S(\cdot, F)$, approximates the desired property $T(F)$ with precision not worse than $\eps$. [\ldots] In some sense order sensitivity is a global and qualitative notion while self-calibration is a local and quantitative notion.''}
\end{quote}
In line with this quotation, self-calibration can be considered as the continuity of the inverse of the expected score $\bar S(\cdot, F)$ at the global minimum $x=T(F)$ -- and as such, it is a local property of the inverse. This property ensures that convergence of the expected score to its global minimum implies convergence of the forecast to the true functional value. 
On the other hand, self-calibration of a scoring function $S$ is equivalent to the fact that the argmin $T(F)$ of the expected score $\bar S(\cdot, F)$ is a \emph{well-separated} point of minimum in the sense of \citet[p.\ 45]{Vaart1998} -- as such being a global property of the expected score itself. That means that for any $\eps>0$
\[
\inf\{\bar S(x,F) \colon \|T(F) - x\| \ge \eps\} > \bar S(T(F),F).
\] 

It is relatively straight forward to see that self-calibration implies strict consistency: Let $S$ be $\F$-self-calibrated for $T$, $F\in\F$, $t=T(F)$ and $x\in\A$ with $x\neq t$. Then for $\eps:=\|x-t\|/2>0$ there is a $\delta>0$ such that $\bar S(x,F) - \bar S(t,F)\ge\delta>0$.

In the preprint version of \cite{SteinwartPasinETAL2014} it is shown for $k=1$ that order-sensitivity implies self-calibration. The next Proposition shows that the kind of order-sensitivity given by Proposition \ref{prop:decr paths} also implies self-calibration for $k\ge1$.

\begin{prop}\label{prop:self-calibration}
Let $\F$ be convex, $\A\subseteq\R^k$ be closed, and $T\colon \F\to\A$ be a surjective and mixture-continuous functional. If $S\colon\A\times \O\to\R$ is strictly $\F$-consistent for $T$ and $\bar S(\cdot,F)\colon \A\to\R$ is continuous for all $F\in\F$, then $S$ is $\F$-self-calibrated for $T$.
\end{prop}

\begin{proof}
Let $F\in \F$, $t=T(F)$ and $\eps>0$. Define 
\[
\delta := \min\{\bar S(z,F) - \bar S(t,F)\colon z\in \A, \ \|z -t\|=\eps\}.
\]
Due to the continuity of $\bar S(\cdot, F)$, the minimum is well-defined and, as a consequence of the strict $\F$-consistency of $S$ for $T$, $\delta$ is positive. 
Let $x\in\A$. If $ \|x-t\|=\eps$, we have, by the definition of $\delta$, that $\bar S(x,F) - \bar S(t,F) \ge \delta$. Assume that $ \|x-t\|>\eps$. Then there is a distribution $G\in\F$ with $T(G) = x$. Due to Proposition \ref{prop:decr paths} there is a continuous path $\gamma\colon[0,1]\to\A$ such that $\gamma(0)=x$, $\gamma(1)=t$ and such that $\bar S(\gamma(\lambda),F)$ is decreasing in $\lambda$. Moreover, if $\lambda<\lambda'$ such that $\gamma (\lambda )\neq \gamma(\lambda')$ it holds that $\bar S(\gamma(\lambda),F)>\bar S(\gamma(\lambda'),F)$. Due to the continuity of $\gamma$ there is some $x'\in \gamma([0,1])$ with $ \|x'-t\|=\eps$. Then we obtain $\bar S(x,F)- \bar S(t,F)>\bar S(x',F)- \bar S(t,F)\ge \delta$.
\end{proof}

We end this subsection about self-calibration by demonstrating its applicability in the context of M-estimation. 

\begin{thm}\label{thm:M-estimation consistency}
Let $S\colon \A\times \O\to\R$ be an $\F$-self-calibrated scoring function for a functional $T\colon\F\to\A\subseteq \R^k$. Then, the following assertion holds for all $F\in\F$.
If $Y_1, Y_2, \ldots$ is a sequence of random variables with distribution $F\in\F$ such that 
\[
\sup_{x\in\A} \left| \frac{1}{n}\sum_{i=1}^n S(x,Y_i) - \bar S(x,F)\right| \stackrel{\P}{\longrightarrow} 0,
\]
then 
\[
\argmin_{x\in\A} \frac{1}{n}\sum_{i=1}^n S(x,Y_i) \stackrel{\P}{\longrightarrow} T(F).
\]
\end{thm}

\begin{proof}
This is a direct consequence of \citet[Theorem 5.7]{Vaart1998}.
\end{proof}

\section{Order-sensitivity} \label{sec:order-sensitivity}

\subsection{Different notions of order-sensitivity}

The idea of order-sensitivity is that a forecast lying between the true functional value and some other forecast is also assigned an expected score lying between the two other expected scores. If the action domain is one dimensional, there are only two cases to consider: both forecasts are on the left-hand side of the functional value or on the right-hand side. However, if $\A\subseteq \R^k$ for $k\ge2$, the notion of `lying between' is ambiguous. Two obvious interpretations for the multidimensional case are the componentwise interpretation and the interpretation that one forecast is the convex combination of the true functional value and the other forecast. 

\begin{defn}[Componentwise order-sensitivity]
A scoring function $S\colon \A\times \O\to\R$ is called componentwise $\F$-order-sensitive for a functional $T\colon \F\to\A\subseteq\R^k$, if for all $F\in \F$, $t =T(F)$ and for all $x, z \in \A$ we have that:
\begin{multline}\label{eqn:componentwise order-sensitive}
\text{For all}\; m \in\{1, \ldots, d\} : z_m\le x_m \le T_m(F) \text{ or } z_m \ge x_m \ge T_m(F)\\
\ \implies \ \bar S(x,F) \le \bar S(z,F).
\end{multline}
Moreover, $S$ is called strictly componentwise $\F$-order-sensitive for $T$ if $S$ is componentwise $\F$-order-sensitive and if $x\neq z$ in \eqref{eqn:componentwise order-sensitive} implies that $\bar S(x,F)<\bar S(z,F)$.
\end{defn}
\begin{rem}\label{rem:Pareto improvements}
In economic terms, a strictly componentwise order-sensitive scoring function rewards \emph{Pareto improvements}\footnote{The definition of the \emph{Pareto principle} according to \cite{ScottMarshall2009}:
``A principle of welfare economics derived from the writings of Vilfredo Pareto, which states that a legitimate welfare improvement occurs when a particular change makes at least one person better off, without making any other person worse off. A market exchange which affects nobody adversely is considered to be a `Pareto-improvement' since it leaves one or more persons better off. `Pareto optimality' is said to exist when the distribution of economic welfare cannot be improved for one individual without reducing that of another.''}
in the sense that improving the prediction performance in one component without deteriorating the prediction ability in the other components results in a lower expected score.
\end{rem}

\begin{defn}[Order-sensitivity on line segments]
Let $\|\cdot\|$ be the Euclidean norm on $\R^k$.
A scoring function $S\colon \A\times \O\to\R$ is $\F$-order-sensitive on line segments for a functional $T\colon \F\to\A\subseteq\R^k$, if for all $F\in\F$, $t = T(F)$, and for all $v\in \mathbb S^{k-1} := \{ x\in\R^k \colon \|x\|=1\}$ the map
\[
\psi\colon D=\{s\in[0,\infty)\colon t+sv\in\A\} \to\R, \quad s\mapsto \bar S(t+sv,F)
\]
is increasing. If the map $\psi$ is strictly increasing, we call $S$ strictly $\F$-order-sensitive on line segments for $T$.
\end{defn}


These two notions of order-sensitivity do not allow for a comparison of any two misspecified forecasts, no matter where they are relative to the true functional value. An intuitive requirement could be `the closer to the true functional value the smaller the expected score', thus calling for the notion of a metric. Since, for a fixed functional $T$ and some fixed distribution $F$, we always have a fixed reference point $T(F)$ and we have the induced vector-space structure of $\R^k$ on $\A$, we shall only work with $\ell^p$-norms $\|\cdot\|_p$, $p\in[1,\infty]$. Recall that for $x\in\R^k$,  $\|x\|_p:=(\sum_{i=1}^k|x_i|^p)^{1/p}$ for $p\in[1,\infty)$ and $\|x\|_\infty := \sup_{i=1, \ldots, k}|x_i|$. If the assertion does not depend on the choice of $p$, we shall usually omit the $p$ in the notation. For other choices of $\A$, it would be also interesting to replace the norm by a metric in the following definition.

\begin{defn}[Metrical order-sensitivity]\label{defn:Metrical order-sensitivity}
Let $p \in [1,\infty]$.
A scoring function $S\colon \A\times \O\to\R$ is \emph{metrically $\F$-order-sensitive} for a functional $T\colon \F\to\A\subseteq\R^k$ relative to the $\ell^p$-norm, if for all $F\in \F$, $t=T(F)$ and for all $x, z\in \A$ we have that 
\begin{align}\label{eqn:metrical order-sensitive}
\|x-t\|_p \le \|z-t\|_p \ &\implies \ \bar S(x,F) \le \bar S(z,F).
\end{align}
If additionally the inequalities in \eqref{eqn:metrical order-sensitive} are strict, we say that $S$ is \emph{strictly metrically $\F$-order-sensitive} for $T$ relative to $\|\cdot\|_p$.
\end{defn}


Similarly to (strict) consistency, all three notions of (strict) order-sensitivity are preserved when considering two scoring functions that are equivalent.\footnote{Two scoring functions $S_1,S_2\colon\A\times \O\to\R$ are \emph{equivalent} if there is a positive constant $\lambda>0$ and an $\F$-integrable function $a\colon\O\to\R$ such that $S_2(x,y) = \lambda S_1(x,y) + a(y)$, for all $(x,y)\in\A\times \O$.}

The notion of componentwise order-sensitivity corresponds almost literally to the notion of accuracy-rewarding scoring functions introduced by \cite{Lambert2008}. Metrically order-sensitivity scoring functions have their counterparts in the field of probabilistic forecasting in effective scoring rules introduced by \cite{Friedman1983} and further investigated by \cite{Nau1985}. Actually, the latter paper has also given the inspiration for the notion of order-sensitivity on line segments. It is obvious that any of the three notions of (strict) order-sensitivity implies (strict) consistency. The next lemma formally states this result and gives some logical implications concerning the different notions of order-sensitivity. The proof is standard and therefore omitted.

\begin{lem}\label{lem:implications}
Let $T\colon \F\to\A\subseteq \R^k$ be a functional and $S\colon\A\times \O\to\R$ a scoring function. 
\begin{enumerate}[\rm (i)]
\item Let $p \in [1,\infty)$.
If $S$ is (strictly) metrically $\F$-order-sensitive for $T$ relative to the $\ell^p$-norm, then $S$ is (strictly) componentwise $\F$-order-sensitive for $T$.
\item[\rm (i')]
If $S$ is (strictly) metrically $\F$-order-sensitive for $T$ relative to the $\ell^\infty$-norm, then $S$ is componentwise $\F$-order-sensitive for $T$.
\item[\rm (i'')]
If $S$ is (strictly) metrically $\F$-order-sensitive for $T$ relative to the $\ell^\infty$-norm, then $S$ is (strictly) $\F$-consistent for $T$.
\item 
If $S$ is (strictly) componentwise $\F$-order-sensitive for $T$, then $S$ is (strictly) $\F$-order-sensitive on line segments for $T$.
\item 
If $S$ is (strictly) $\F$-order-sensitive on line segments for $T$, then $S$ is (strictly) $\F$-consistent for $T$.
\end{enumerate}
\end{lem}

\subsection{Componentwise order-sensitivity}\label{subsec:Componentwise}

Under restrictive regularity assumptions, \citet[Theorem 5]{Lambert2008} claim that whenever a functional has a componentwise order-sensitive scoring function, the components of the functional must be elicitable. Moreover, assuming that the measures in $\F$ have finite support, they assert that any componentwise order-sensitive scoring function is the sum of strictly consistent scoring functions for the components. 
Lemma \ref{lem:components} shows the first claim under less restrictive smoothness assumptions on the scoring function. For many common examples of functionals, the second claim can be shown relaxing the restrictive condition on $\F$; see Proposition \ref{prop:sum} and the discussion before.

\begin{lem}\label{lem:components}
Let $T = (T_1, \ldots, T_k)\colon \F\to\A\subseteq \R^k$ be a $k$-dimensional functional with components $T_m\colon \F\to\A_m\subseteq \R$ where $\A = \A_1\times \cdots\times \A_k$. If there is  a strictly component\-wise $\F$-order-sensitive scoring function $S\colon \A\times \O\to\R$ for $T$, then the components $T_m$, $m\in\{1, \ldots, k\}$, are elicitable.
\end{lem}

\begin{proof}
Fix $m\in\{1, \ldots, k\}$. Let $F\in\F$ and $x,z\in\A$ such that $T_m(F) = x_m$, $x_i = z_i$ for all $i\neq m$ and $x_m\neq z_m$.
Due to the strict componentwise $\F$-order-sensitivity of $S$ this implies that 
$\bar S(x,F)<\bar S(z,F)$.
This in turn means that for any $z = (z_1, \ldots, z_k)\in\A$ the map $S_{m,z}\colon\A_m \times \O \to\R$,
\begin{align}\label{eq:S_{m,z}}
&(x_m,y)\mapsto S_{m,z}(x_m,y):=S(z_1, \ldots, z_{m-1},x_m,z_{m+1}, \ldots, z_k,y)
\end{align}
is a strictly $\F$-consistent scoring function for $T_m$.
\end{proof}

If $T_m\colon\F\to\A_m\subseteq \R$, $m\in\{1, \ldots k\}$, are mixture-continuous and elicitable with strictly $\F$-consistent scoring functions $S_m\colon\A_m\times \O\to\R$, then they are order-sensitive according to \citet[Proposition 2]{Lambert2013} and \citet[Proposition 3.4]{BelliniBignozzi2015}. Therefore, the sum $\sum_{m=1}^k S_m(x_m,y)$ is strictly componentwise $\F$-order-sensitive for $(T_1, \ldots, T_k)$.
More interestingly, one can establish the reverse of the last assertion.
Any strictly componentwise order-sensitive scoring function must necessarily be additively separable. In \citet[Section 4]{FisslerZiegel2016}, we established a dichotomy for functionals with elicitable components: In most relevant cases, the functional (the corresponding strict identification function, respectively) satisfies Assumption (V4) therein (e.g., when the functional is a vector of different quantiles and\,/\,or different expectiles with the exception of the 1/2-expectile), or it is a vector of ratios of expectations with the same denominator, or it is a combination of both situations. Under some regularity conditions, \citet[Propositions 4.2 and 4.4]{FisslerZiegel2016} characterize the form of strictly consistent scoring functions for the first two situations, whereas \citet[Remark 4.5]{FisslerZiegel2016} is concerned with the third situation. For this latter situation, any strictly consistent scoring function must be necessarily additive for the respective blocks of the functional. And for the first situation, \citet[Proposition 4.2]{FisslerZiegel2016} yields the additive form of $S$ automatically. It remains to consider the case of \citet[Proposition 4.4]{FisslerZiegel2016}, that is, a vector of ratios of expectations with the same denominator.

\begin{prop}\label{prop:sum}
Let $T\colon\F\to\A\subseteq \R^k$ be a ratio of expectations with the same denominator, that is, $T(F) = \E_F[p(Y)] / \E_F[q(Y)]$ for some $\F$-integrable functions $p\colon\O\to\R^k$, $q\colon \O\to\R$ such that $\E_F[q(Y)]>0$ for all $F\in\F$.\footnote{It is no loss of generality to assume that $\bar q(F)>0$ for all $F\in\F$ in Proposition \ref{prop:sum}. In order to ensure that $T$ is well-defined, necessarily $\bar q(F)\neq 0$ for all $F\in\F$. However, Assumption (V1) implies that $\F$ is convex. So if there are $F_1,F_2\in\F$ such that $\bar q(F_1)<0$ and $\bar q(F_2)>0$ then there is a convex combination $G$ of $F_1$ and $F_2$ such that $\bar q (G)=0$. Consequently, either $\bar q(F)>0$ for all $F\in\F$ or $\bar q(F)<0$ for all $F\in\F$, and by possibly changing the sign of $p$ one can assume that the first case holds.} 
Assume that $T$ is surjective, and that $\interior(\A)\neq\emptyset$ is simply connected. Moreover, consider the strict $\F$-identification function $V\colon \A\times \O\to\R^k$, $V(x,y) = q(y) x - p(y)$ and some strictly $\F$-consistent scoring function $S\colon \A\times \O\to\R$ such that the Assumptions (V1), (S2), (F1), and (VS1) in \cite{FisslerZiegel2016} hold. If $S$ is strictly componentwise $\F$-order-sensitive for $T$, then $S$ is of the form
\be{eq:S sum}
S(x_1, \ldots, x_k,y) = \sum_{m=1}^k S_m(x_m,y),
\ee
for almost all $(x,y) \in \A \times \O$, where $S_m\colon\A_m\times \O\to\R$, $m\in\{1, \ldots, k\}$, are strictly $\F$-consistent scoring functions for $T_m\colon \F\to \A_m$, $\A_m:=T_m(\F)\subseteq \R$, and $T_m(F) = \E_F[p_m(Y)]/\E_F[q(Y)]$.
\end{prop}

\begin{proof}
Due to the fact that for fixed $y\in\O$, $V(x,y)$ is a polynomial in $x$, Assumption (V3) in \cite{FisslerZiegel2016} is automatically satisfied.
Let $h\colon\interior(\A)\to\R^{k\times k}$ be the matrix-valued function given in Osband's principle; see \citet[Theorem 3.2]{FisslerZiegel2016}. By \citet[Proposition 4.4(i)]{FisslerZiegel2016} we have that 
\begin{align}\label{eq:cyclic}
\partial_l h_{rm}(x) = \partial_r h_{lm}(x), && h_{rl}(x) = h_{lr}(x)
\end{align}
for all $r,l,m\in\{1,\ldots, k\}$, $l\neq r$, where the first identity holds for almost all $x\in \interior(\A)$ and the second identity for all $x\in\interior(\A)$.
Moreover, the matrix $\big(h_{rl}(x)\big)_{l,r=1, \ldots, k}$ is positive definite for all $x\in\interior(\A)$. If we can show that $h_{lr}=0$ for $l\neq r$, we can use the first part of \eqref{eq:cyclic} and deduce that for all $m\in\{1,\ldots, k\}$ there are positive functions $g_m\colon \A_m'\to\R$, where $\A_m'= \{x_m\in\R\colon \exists  (z_1, \ldots, z_k)\in\interior(\A) \text{ and }z_m = x_m\}$, such that 
\[
h_{mm}(x_1, \ldots, x_k) = g_m(x_m)
\]
for all $(x_1, \ldots, x_k)\in\interior(\A)$. Then, we can conclude like in the proof of \citet[Proposition 4.2(ii)]{FisslerZiegel2016}.\footnote{The arguments in \citet[Proposition 4.2(ii)]{FisslerZiegel2016} use \citet[Proposition 3.4]{FisslerZiegel2016}. There is a flaw in the latter result which has been pointed out in \cite{Brehmer2017}. We present a corrected version of the result in Appendix \ref{appendix:A}.
}

Fix $l,r\in\{1, \ldots, k\}$ with $l\neq r$ and $F\in\F$ such that $T(F)\in\interior(\A)$.
Due to the strict $\F$-consistency of $S_{l,z}$ defined at \eqref{eq:S_{m,z}} we have that
\[
0=\frac{\dint}{\dint x_l}\bar S_{l,z}(x_l,F) = \partial \bar S_{l,z}(x_l,F)= \partial_l \bar S(z_1, \ldots, z_{l-1},x_l,z_{l+1}, \ldots, z_k,F)
\]
whenever $x_l = T_l(F)$ and for all $z\in\interior(\A)$. This means the map $\interior(\A)\ni z\mapsto \partial \bar S_{l,z}(T_l(F),F)$ is constantly 0. Hence, for all $x\in\interior(\A)$ 
\[
\partial_r\partial_l \bar S(x,F)=0
\]
whenever $x_l = T_l(F)$. Using the special form of $V$ and \citet[Corollary 3.3]{FisslerZiegel2016}, we have for $x=t=T(F)$ that 
\[
0= \partial_r\partial_l \bar S(t,F) = h_{lr}(t) \partial_r \bar V_r(t,F) =  h_{lr}(t)\bar q(F)
\]
and by assumption $\bar q(F)>0$. Using the surjectivity of $T$ we obtain that $h_{lr}(t) = 0$ for all $t\in\interior(\A)$, which ends the proof.
\end{proof}


The notion of componentwise order-sensitivity has an appealing interpretation in the sense that it rewards Pareto improvements of the predictions; see Remark \ref{rem:Pareto improvements}. The results of Lemma \ref{lem:components} and Proposition \ref{prop:sum} give a clear understanding of the concept including its limitations to the case of functionals only consisting of elicitable components.

\citet{EhmETAL2016} introduced Murphy diagrams for forecast comparison of quantiles and expectiles. Murphy diagrams have the advantage that forecasts are compared simultaneously with respect to all consistent scoring functions for the respective functional. For many multivariate functionals such as ratios of expectations, the methodology cannot be readily extended because there are no mixture representations available for the class of all consistent scoring functions. Proposition \ref{prop:sum} shows that when considering only componentwise order-sensitive consistent scoring functions, the situations is different and mixture representations (and hence Murphy diagrams) are readily available for forecast comparison.

\subsection{Metrical order-sensitivity}

We start with an equivalent formulation of metrical order-sensitivity. 

\begin{lem}\label{lem:metrical vs symmetry2}
Let $\F$ be convex and $T\colon\F\to\A\subseteq \R^k$ be mixture-continuous and surjective. Let $S\colon \A\times \O\to\R$ be (strictly) $\F$-consistent for $T$. Then $S$ is (strictly) metrically $\F$-order-sensitive for $T$ relative to $\|\cdot\|$ if and only if for all $F\in\F$, $t=T(F)$ and $x,z\in\A$ we have the implication
\begin{align}\label{eq:symmetry2}
\|x-t\| = \|z-t\| \ \implies \ \bar S(x,F) = \bar S(z,F).
\end{align}
\end{lem}
\begin{proof}
Let $S$ be metrically $\F$-order sensitive for $T$ relative to $d$. Let $F\in\F$, $t=T(F)$, $x,z\in\A$ such that $\|x-t\| = \|z-t\|$. Then we have both $\bar S(x,F)\le \bar S(z,F)$ and $\bar S(z,F)\le \bar S(x,F)$.

Assume that \eqref{eq:symmetry2} holds and $S$ is (strictly) $\F$-consistent. Let $F\in\F$ with $t=T(F)$ and $x,z\in\A$. Suppose that $\|x-t\| \le \|z-t\|$. If $\|x-t\| = \|z-t\|$, \eqref{eq:symmetry2} implies that $\bar S(x,F) = \bar S(z,F)$ and there is nothing to show. If $\|x-t\|< \|z-t\|$, we can apply Proposition \ref{prop:decr paths}. There is a continuous path $\gamma\colon [0,1]\to\A$ such that $\gamma(0) = z$ and $\gamma(1)=t$, and the function $[0,1]\ni\lambda \mapsto \bar S(\gamma(\lambda), F)$ is decreasing. Due to continuity there is a $\lambda'\in[0,1]$ such that $\|\gamma(\lambda') - t\| = \|x-t\|$. Invoking \eqref{eq:symmetry2} it holds that $\bar S(x,F) = \bar S(\gamma(\lambda'), F)\le \bar S(z,F)$. If $S$ is strictly $\F$-consistent then the latter inequality is strict.
\end{proof}

For a real-valued functional $T$ there can be at most one strictly metrically order-sensitive scoring function, up to equivalence. To show this, we use Osband's principle and impose the corresponding regularity conditions.

\begin{prop}\label{prop:unique}
Let $T\colon \F\to\A\subseteq\R$ be a surjective, elicitable and identifiable functional with an oriented strict $\F$-identification function $V\colon \A\times \O\to\R$. If $\interior(\A) \neq \emptyset$ is convex and $S, S^*\colon \A\times \O\to\R$ are two strictly metrically $\F$-order-sensitive scoring functions for $T$ such that the Assumptions (V1), (V2), (S1), (F1) and (VS1) from \cite{FisslerZiegel2016} (with respect to both scoring functions) hold, then $S$ and $S^*$ are equivalent almost everywhere. 
\end{prop}
\begin{proof}
We apply Osband's principle, that is, \citet[Theorem 3.2]{FisslerZiegel2016} to $S$. Consequently, there is a function $h\colon\interior(\A)\to\R$ such that 
\be{eq:Osband S_1}
\frac{\diff}{\diff x}\bar S(x,F) = h(x) \bar V(x,F)
\ee
for all $F\in\F$, $x\in\interior(\A)$. Due to the strict $\F$-consistency of $S$ and the orientation of $V$, it holds that $h\ge0$. We show that actually $h>0$. Applying Lemma \ref{lem:metrical vs symmetry2}, one has that 
\be{eq:Sym}
\bar S(T(F)+x, F) = \bar S(T(F)-x,F)
\ee
for all $F\in\F$, $x\in\R$ such that $T(F)+x, T(F)-x\in\interior(\A)$. Hence, also the derivative with respect to $x$ of the left-hand side of \eqref{eq:Sym} must coincide with the derivative on the right-hand side. This yields, using \eqref{eq:Osband S_1}, 
\be{eq:Sym V}
h(T(F) + x)\bar V(T(F)+x,F) = -h(T(F)- x) \bar V(T(F)-x,F)
\ee
for all $F\in\F$, $x\in\R$ such that $T(F)+x, T(F)-x\in\interior(\A)$. Assume $h(z) = 0$ for some $z\in\interior(\A)$. Then, by surjectivity of $T$ and convexity of $\interior(\A)$, for all $z'\in\interior(\A)\setminus\{z\}$ there exists an $F\in \F$ and $x\in\R\setminus\{0\}$ such that $z=T(F)+x$ and $z'=T(F) - x$. Since $V$ is a \emph{strict} $\F$-identification function for $T$, both $\bar V(T(F) + x,F)\neq 0$ and $\bar V(T(F) - x,F)\neq0$. Hence, \eqref{eq:Sym V} implies that $h(z')=0$. This implies that $h$ identically vanishes on $\interior(\A)$ which contradicts the strict $\F$-consistency of $S$. 

Therefore, $V^*(x,y) := h(x)V(x,y)$ is an oriented strict $\F$-identification function for $T$. Applying Osband's principle to $S^*$, one obtains a function $h^*\colon \interior(\A)\to\R$ such that 
$
\diff/(\diff x)\bar S^*(x,F) = h^*(x) \bar V^*(x,F)
$
for all $F\in\F$, $x\in\R$ such that $T(F)+x, T(F)-x\in\interior(\A)$. Due to the analogue of \eqref{eq:Sym} for $S^*$ and \eqref{eq:Sym V}, one obtains
\begin{align*}
h^*(T(F) + x)\bar V^*(T(F)+x,F) &= -h^*(T(F)- x) \bar V^*(T(F)-x,F) \\
&= h^*(T(F)- x) \bar V^*(T(F)+x,F).
\end{align*}
for all $F\in\F$, $x\in\R$ with $T(F)+x, T(F)-x\in\interior(\A)$.
By a similar reasoning as above, one can deduce that $h^*$ must be constant and positive. Now, the claim follows by \citet[Proposition 3.4]{FisslerZiegel2016}; see Appendix \ref{appendix:A}.\end{proof}

For the higher-dimensional setting we can show a slightly more limited version of Proposition \ref{prop:unique}. Two scoring functions that are additively separable as in \eqref{eq:S sum} and that are strictly metrically order-sensitive for the same functional must necessarily be equivalent. 
For most practically relevant cases -- namely when we consider an $\ell^p$-norm with $p\in[1,\infty)$ and when the functional possesses an identification function satisfying Assumption (V4) in \cite{FisslerZiegel2016} or that are ratios of expectations with the same denominator -- Lemma \ref{lem:implications}, Proposition \ref{prop:sum} and \citet[Proposition 4.2]{FisslerZiegel2016} yield that any metrically order-sensitive scoring function -- presuming there is one -- is additively separable. Hence, for these situations, metrically order-sensitive scoring functions are unique, up to equivalence.

\begin{prop}\label{prop:unique lambda}
Let $S\colon\A\times \O\to\R$ be a strictly metrically $\F$-order-sensitive scoring function for a surjective functional $T = (T_1, \ldots, T_k)\colon \F\to\A\subseteq\R^k$ of the form
\[
S(x_1, \ldots, x_k,y)  = \sum_{m=1}^k S_m(x_m,y)
\] 
for all $(x,y)\in\A\times\O$ where $S_m\colon \A_m\times \O\to\R$, $m\in\{1, \ldots, k\}$, $\A_m= \{x_m\in\R\colon \exists  (z_1, \ldots, z_k)\in\A \text{ and }z_m = x_m\}$, are strictly $\F$-consistent scoring functions for $T_m$. 
Assume that $\interior(\A) \neq \emptyset$.
Then, the following assertions hold:
\begin{enumerate}[\rm (i)]
\item
The scoring functions $S_m$, $m\in\{1, \ldots, k\}$, are strictly metrically $\F$-order-sensitive for $T_m$. 
\item
Let $\lambda_1, \ldots, \lambda_k>0$ and define the scoring function $S^*\colon \A\times \O\to\R$ via
\[
S^*(x_1, \ldots, x_k)  = \sum_{m=1}^k \lambda_m S_m(x_m,y).
\]
Then $S^*$ is strictly metrically $\F$-order-sensitive (with respect to the same $\ell^p$-norm as $S$) if and only if $\lambda_1 = \cdots = \lambda_k$.
\end{enumerate}
\end{prop}

\begin{proof}
\begin{enumerate}[(i)]
\item
Let $m\in\{1, \ldots, k\}$, $F\in\F$ with $t = T(F) \in \interior(\A)$. Let $\mu\in\R$ and $x,z\in\interior(\A)$ with $x_i = z_i =t_i$ for all $i\neq m$ and with $x_m = t_m + \mu$ and $z_m = t_m -\mu$, such that $|x-t| = |z-t|$. Due to Lemma \ref{lem:metrical vs symmetry2} and due to the particular additive form of $S$, we have
\begin{align*}
0= \bar S(x,F) - \bar S(z,F) &= \bar S_m(x_m,F) - \bar S_m(z_m,F) \\
&= \bar S_m(t_m + \mu,F) - \bar S_m(t_m-\mu,F).
\end{align*}
Again with Lemma \ref{lem:metrical vs symmetry2} one obtains the assertion.
\item
The only interesting direction is to assume that $S^*$ is strictly metrically $\F$-order-sensitive (with respect to the same $\ell^p$-norm as $S$). We will show that $\lambda_1 = \lambda_m$ for all $m\in\{2, \ldots, k\}$. 
Let $F\in\F$, $t=T(F) \in \interior(\A)$, $x,z\in\interior(\A)$ with $\|x-t\|_p = \|z-t\|_p>0$ and $x_i = z_i = t_i$ for all $i\in\{2, \ldots, k\}\backslash\{m\}$. Moreover, let $x_1\neq z_1 = t_1$. Due to Lemma \ref{lem:metrical vs symmetry2} we have that $\bar S(x,F) - \bar S(z,F) = \bar S^*(x,F) - \bar S^*(z,F) = 0$.
Moreover, 
\begin{align*}
0=\bar S(x,F) - \bar S(z,F) &= \sum_{i=1}^k \bar S_i(x_i,F) - \bar S_i(z_i,F) \\
&= \bar S_1(x_1,F) - \bar S_1(z_1,F) + \bar S_m(x_m,F) - \bar S_m(z_m,F).
\end{align*}
Setting $\varepsilon := \bar S_1(x_1,F) - \bar S_1(z_1,F)>0$, one obtains with the same calculation
\begin{align*}
0&= \bar S^*(x,F) - \bar S^*(z,F) \\
&=\lambda_1\big( \bar S_1(x_1,F) - \bar S_1(z_1,F)\big) + \lambda_m\big(\bar S_m(x_m,F) - \bar S_m(z_m,F)\big) = \varepsilon(\lambda_1 - \lambda_m).
\end{align*}
\end{enumerate}
\end{proof}

Next, we use the derived theoretical results to examine when some popular functionals admit strictly metrically order-sensitive scoring functions, and if so, of what form they are.

\subsubsection{Ratios of expectations with the same denominator}\label{subsubsec:ratio of exp}

We start with the one-dimensional characterization. 

\begin{lem}\label{lem:ratios of exp k=1}
Let $\F$ be convex and $p,q\colon\O\to\R$ two $\F$-integrable functions such that $\bar q(F)>0$ for all $F\in\F$. Define $T\colon\F\to\A\subseteq\R$, $T(F) = \bar p(F) / \bar q(F)$ and assume that $T$ is surjective and $\interior(\A) \neq\emptyset$ is convex. Then the following two assertions are true:
\begin{enumerate}[\rm (i)]
\item
Any scoring function which is equivalent to
\be{eq:ratio_a}
S\colon \A\times \O\to\R, \qquad S(x,y) = \frac12 q(y)x^2 - p(y)x
\ee
is strictly metrically $\F$-order-sensitive for $T$.
\item
If $\F$ is such that Assumptions (V1), (F1) in \cite{FisslerZiegel2016} are satisfied with $V(x,y) = q(y)x-p(y)$, then any scoring function $S^*\colon\A\times\O\to\R$, which is strictly metrically $\F$-order-sensitive  and satisfies Assumptions (S1) and (VS1), is equivalent to $S$ defined at \eqref{eq:ratio_a} almost everywhere.
\end{enumerate}
\end{lem}
\begin{proof}
\begin{enumerate}[\rm (i)]
\item
We can apply Lemma \ref{lem:metrical vs symmetry2}. Let $F\in\F$. Then
\[
\R\ni x\mapsto \bar S(T(F) +x,F) = \frac{1}{2}\bar q(F) x^2 - \frac12 \frac{\bar p(F)^2}{\bar q(F)}
\]
is an even function in $x$. Moreover, equivalence of scoring functions preserves (strict) metrical order-sensitivity.
\item
The convexity of $\A$ is implied by the mixture-continuity of $T$ and the convexity of $\F$. Then, the claim follows with Proposition \ref{prop:unique}.
\end{enumerate}
\end{proof}

Now, we turn to the multivariate characterization.

\begin{prop}\label{prop:multratio}
Let $k\ge2$, $\F$ be convex and $p\colon\O\to\R^k$, $q\colon\O\to\R$ two $\F$-integrable functions such that $\bar q(F)>0$ for all $F\in\F$. Define $T\colon\F\to\A\subseteq\R^k$, $T(F) = \bar p(F) / \bar q(F)$ and assume that $T$ is surjective and $\interior(\A)\neq \emptyset$. Then, the following assertions are true:
\begin{enumerate}[\rm (i)]
\item
Any scoring function which is equivalent to
\be{eq:ratio2_a}
S\colon \A\times \O\to\R, \qquad S(x_1, \ldots, x_k,y) =\sum_{m=1}^k  \frac12 q(y)x_m^2 - p_m(y)x_m
\ee
is strictly metrically $\F$-order-sensitive for $T$ with respect to the $\ell^2$-norm.
\item
If $\F$ is such that Assumptions (V1), (F1) in \cite{FisslerZiegel2016} are satisfied with $V(x,y) = q(y)x-p(y)$, then any scoring function $S^*\colon\A\times\O\to\R$, which is strictly metrically $\F$-order-sensitive with respect to the $\ell^2$-norm and satisfies Assumptions (S1) and (VS1), is equivalent to $S$ defined at \eqref{eq:ratio2_a} almost everywhere.
\item
If $\F$ is such that Assumptions (V1), (F1) in \cite{FisslerZiegel2016} are satisfied with $V(x,y) = q(y)x-p(y)$, then there is no scoring function $S^*\colon\A\times\O\to\R$ which satisfies Assumptions (S1) and (VS1) and which is strictly metrically $\F$-order-sensitive with respect to an $\ell^p$-norm with $p\in[1,\infty)\setminus\{2\}$.
\end{enumerate}
\end{prop}
\begin{proof}
To show (i) we apply again Lemma \ref{lem:metrical vs symmetry2}. For any $F\in\F$, $x \in \R^k$, we have $\bar S(T(F) + x, F) = (1/2) \bar q(F)\|x\|_2^2  - 1/(2\bar q(F))\sum_{m=1}^k \bar p_m(F)^2$ which only depends on the $\ell^2$-norm of $x$.

We prove (ii) and (iii) together. Assume there is a scoring function $S^*$ satisfying the conditions above, so in particular, it is strictly metrically $\F$-order-sensitive with respect to the $\ell^p$-norm for $p\in[1, \infty)$. Invoking Lemma \ref{lem:implications}(i), $S^*$ is strictly componentwise $\F$-order-sensitive for $T$. Thanks to Proposition \ref{prop:sum}, $S^*$ is additively separable. By Lemma \ref{lem:ratios of exp k=1}(i), it is of the form 
\[
S^*(x_1,\dots,x_k) = \sum_{m=1}^k\lambda_m \left(\frac{1}{2}q(y)x_m^2 - p_m(y)x_m\right) + \sum_{m=1}^k a_m(y).
\]
If $p=2$, part (i) and Proposition \ref{prop:unique lambda}(ii) yield that $\lambda_1 = \dots =\lambda_k$, and hence, $S$ and $S^*$ are equivalent.
For $p\not=2$, we obtain $\bar S(T(F) + x, F) = (1/2) \bar q(F)\sum_{m=1}^k \lambda_m x_m^2  - 1/(2\bar q(F))\sum_{m=1}^k \bar p_m(F)^2$. It is not hard to see that there are always $x,x'$ with $\|x\|_p = \|x'\|_p$ but $\bar S(T(F) + x, F) \not= \bar S(T(F) + x', F)$.
\end{proof}

\citet[Section 5]{Savage1971} has already shown that in case of the mean, the squared loss is essentially the only symmetric loss in the sense that it is the only metrically order-sensitive loss for the mean. See also \citet[Section 2.1]{Patton2016} for a discussion that symmetry -- or metrical order-sensitivity -- is not necessary for strict consistency of scoring functions with respect to the mean.

\subsubsection{Quantiles}

Since we treat only point-valued functionals in this article, we shall assume that the $\a$-quantile of $F$ is a singleton and identify the set with its unique element (henceforth, we shall refer to this assumption as $F$ having a unique $\a$-quantile).\footnote{Recall that the $\a$-quantile of a distribution $F$ consists of all points $x\in\R$ satisfying $\lim_{t\uparrow x}F(t)\le \a\le F(x)$.}
Furthermore, note that assuming the identifiability of the $\a$-quantile with the canonical identification function $V_\alpha(x,y) = \one\{y\le x\} - \alpha$ on a class $\F$ amounts to assuming that $F(q_\a(F))=\a$ for all $F\in\F$.\footnote{Actually, assuming $\F$ is convex and rich enough, this holds for any identification function for the $\a$-quantile. Indeed, consider some distribution function $F_0\in\F$ and some level $\a\in(0,1)$. Fix some $x_0\in\R$ such that $F_0(x_0)<\a$, implying that $q_\a(F_0)>x_0$. Assume that for any $\lambda\in[0,1]$, the distribution 
\[
F_\lambda(x) =\begin{cases}
F_0(x), & x<x_0 \\
(1-\lambda)F_0(x) + \lambda, & x\ge x_0
\end{cases}
\]
is an element of $\F$. Then, there is some $\lambda'\in(0,1)$ such that $F_{\lambda'}(x_0) = \a$ implying that $F_\lambda(x_0)>\a$ for all $\lambda\in(\lambda',1]$ and $q_\a(F_{\lambda}) = x_0$ for all $\lambda\in[\lambda',1]$. Assume that $V$ is a strict $\F$-identification function for $q_\a$. That means $\bar V(x_0,F_{\lambda})=0$ for all $\lambda\in[\lambda',1]$ and $\bar V(x_0,F_{\lambda})\neq0$ for all $\lambda \in[0,\lambda')$. Consider some $\lambda \in[\lambda',1]$. 
Then, 
\[
\bar V(x_0,F_\lambda) = (1-\lambda)\bar V(x_0,F_0) + \lambda \bar V(x_0,F_1) = (1-\lambda)\bar V(x_0,F_0)\neq0.
\]
This is a contradiction to $V$ being a strict $\F$-identification function for $q_\a$.}
\begin{prop}\label{prop:quantiles and metric order sensitivity}
Let $\alpha\in(0,1)$ and $\F$ be a family of distribution functions $F$ on $\R$ with unique $\alpha$-quantiles $q_\alpha(F)$ satisfying $F(q_\a(F)) = \a$ for all $F\in\F$.
Assume that there is an $F_0\in\F$, such that its translation $F_\lambda(\cdot) = F_0(\cdot - \lambda)$ is also an element of $\F$ for all $\lambda\in\R$. Consequently, $T_\alpha = q_\a\colon \F\to\A=\R$ is surjective. Under assumptions (V1) in \cite{FisslerZiegel2016} with respect to the strict identification function $V_\alpha\colon \R\times\R\to\R$, $V_\alpha(x,y) = \one\{y\le x\} - \alpha$, there is no strictly metrically $\F$-order-sensitive scoring function for $T_\alpha$ satisfying Assumption (S1) in \cite{FisslerZiegel2016}.
\end{prop}
\begin{proof}
Assume that there exists a strictly metrically $\F$-order-sensitive scoring function $S_\alpha\colon \R\times \R\to\R$ satisfying Assumption (S1) in \cite{FisslerZiegel2016}. 
Due to Lemma \ref{lem:metrical vs symmetry2}, for any $F\in\F$ and any $x\in\R$
\[
\bar S_\alpha(T_\alpha(F) + x,F) = \bar S_\alpha(T_\alpha(F) - x,F). 
\]
Using Osband's principle \cite[Theorem 3.2]{FisslerZiegel2016} and taking the derivative with respect to $x$ on both sides, this yields
\begin{align}\label{eq:identity quantiles}
h(T_\alpha(F) + x) \bar V_\alpha(T_\alpha(F) +x,F) = -h(T_\alpha(F) -x) \bar V_\alpha(T_\alpha(F) -x,F)
\end{align}
for some positive function $h\colon \R\to\R$ (the fact that $h\ge0$ follows from the strict consistency of $S_\alpha$ and the surjectivity of $T_\alpha$, and $h>0$ follows like in the proof of Proposition \ref{prop:unique}). Assume that $T_\alpha(F_0)=0$. For $\lambda \in \R$, we have $T_\alpha(F_0(\cdot - \lambda)) = \lambda$. Therefore, \eqref{eq:identity quantiles} implies
\begin{equation}\label{eq:ratio2}
\frac{h(\lambda+ x)}{h(\lambda - x)} =  - \frac{\bar V_\alpha(\lambda - x,F_0(\cdot - \lambda))}{\bar V_\alpha(\lambda + x,F_0(\cdot - \lambda))}
=- \frac{F_0(- x) -\alpha}{F_0( x) -\alpha}.
\end{equation}

Setting $\lambda = \pm x$, one can see that $h(\pm\infty):=\lim_{x\to\pm\infty}h(x)$ exists and that $h(+\infty) = h(0)\alpha/(1-\alpha)$, $h(-\infty) = h(0)(1-\alpha)/\alpha$, hence, $h(+\infty)/h(-\infty) = 1$. On the other hand, for fixed $\lambda \in \R$, we obtain
\[
\frac{h(+\infty)}{h(-\infty)} = \lim_{x\to\infty} \frac{h(\lambda+ x)}{h(\lambda - x)} = \frac{\alpha}{1-\alpha}.
\]
As a consequence, the only remaining possibility is $\alpha=1/2$. For fixed $x \in \R$, we have
\[
1= \frac{h(+\infty)}{h(+\infty)} = \lim_{\lambda\to\infty} \frac{h(\lambda+ x)}{h(\lambda - x)} 
=- \frac{F_0(- x) - 1/2}{F_0( x) - 1/2}
\]
implying that $h$ must be constant using \eqref{eq:identity quantiles}, and that $F_0$ must be symmetric around its median, i.e.~$F_0(x) = 1 - F_0(-x)$ for all $x\in\R$.\footnote{This equation implies that $F_0$ is necessarily continuous. This fact also follows directly from Assumption (S1) in \cite{FisslerZiegel2016} and the assumption that $\F$ is closed under translations of $F_0$. Indeed, assume that $F_0$ is discontinuous at some point $x_0$. Then $h$ has to be discontinuous at that point. But since $F_0$ has at most countably many points of discontinuity, there is some $\lambda_0\in\R$ such that $F_{\lambda_0}$ is continuous at $x_0$. But this would imply that the derivative of $\bar S(\cdot,F_{\lambda_0})$ is discontinuous at $x_0$, which contradicts the assumptions. 
 }
Moreover, since $h$ is constant, \eqref{eq:identity quantiles} implies that also any other distribution $F\in\F$ must be symmetric around its median, i.e.~$F(T_{1/2}(F)+x) = 1 - F(T_{1/2}(F) - x)$ for all $x \in \R$. However, if $F_0$ is symmetric around its median, then any translation $F_\lambda$ of $F_0$ is symmetric around its median. But then, there is a convex combination of $F_0$ and $F_\lambda$ with mixture-parameter $\beta\in(0,1)$, $\beta\neq1/2$, such that $\beta F_0 + (1-\beta)F_\lambda$ is not symmetric around its median if $\lambda\neq0$. Consequently, the conditions of the proposition are violated such that a strictly metrically $\F$-order-sensitive function for the median does not exist in this setting.
\end{proof}


The reasons for the non-existence of a strictly metrically order-sensitive scoring function for the $\alpha$-quantile are of different nature in the two cases that $\alpha\neq1/2$ and that $\alpha=1/2$ in the proof of Proposition \ref{prop:quantiles and metric order sensitivity}. In both cases, we used Osband's principle to derive a representations of the derivative of the expected score. Assuming that the derivative has the form as stated in Osband's principle, one can directly derive a contradiction for $\alpha\neq1/2$. However, for $\alpha=1/2$, this form merely implies that the distributions in $\F$ must be symmetric around their medians. This is not contradictory to the form of the gradient derived via Osband's principle, but only to the assumption that $\F$ is convex. Dropping this assumption, we can derive the following Lemma. The proof is straight forward from Lemma \ref{lem:metrical vs symmetry2}. 

\begin{lem}\label{lem:absolute loss}
Let $\F$ be a family of distribution functions on $\R$ with unique medians $T_{1/2}\colon\F\to\R$ and finite first moments. 
If all distributions in $\F$ are symmetric around their medians in the sense that 
\be{eq:symmetry around median3}
F(T_{1/2}(F) +x) = 1 - F((T_{1/2}(F) - x)-)
\ee
for all $F\in\F$, $x\in\R$, then any scoring function that is equivalent to the absolute loss $S\colon\R\times\R\to\R$, $S(x,y) = |x-y|$, is strictly metrically $\F$-order-sensitive with respect to the median.
\end{lem}


As mentioned above, under the conditions of Lemma \ref{lem:absolute loss}, the necessary characterization of strictly consistent scoring functions via Osband's principle is not available. In particular, this means that we cannot use Proposition \ref{prop:unique}. Indeed, if the distributions in $\F$ are symmetric around their medians in the sense of \eqref{eq:symmetry around median3} and under the integrability condition that all elements in $\F$ have a finite first moment, the median and the mean coincide. Hence, any convex combination of a strictly consistent scoring function for the mean and the median provides a strictly consistent scoring function.
\emph{A fortiori}, any scoring function which is equivalent to $S(x,y) = (1-\lambda) |x-y| + \lambda|x-y|^2$, $\lambda\in[0,1]$ is strictly metrically $\F$-order-sensitive.
However, the class of strictly metrically $\F$-order-sensitive scoring functions is even bigger -- \citet[Corollary 7.19, p.~50]{LehmannCasella1998} show that (subject to integrability conditions) for an even and strictly convex function $\Phi\colon\R\to\R$, the score $S(x,y) = \Phi(x-y)$ is strictly metrically $\F$-order-sensitive for the median. Note that if the distributions in $\F$ are symmetric, their \emph{center of symmetry}, which is the functional solving \eqref{eq:symmetry around median3}, is unique \cite[Lemma 4.1.34]{Fissler2017}, even if the median is not unique. The result of \citet[Corollary 7.19, p.~50]{LehmannCasella1998} holds for this center of symmetry. Acknowledging that some popular choices for $\Phi$ are not strictly convex (see Example \ref{exmp: Huber loss}), the following proposition gives a refinement of their result.

\begin{prop}
Let $\F$ be a class of symmetric distributions on $\R$ with center of symmetry $C\colon\F\to\R$, that is, $F(C(F) +x) = 1 - F((C(F) - x)-)$ for all $F\in \F$, $x \in \R$. Let $\Phi\colon\R\to\R$ be a convex and even function, and $S\colon\R\times\R\to\R$, $S(x,y) = \Phi(x-y)$. For any $x\in\R$, define the function $\Psi_x\colon\R\to\R$, $\Psi_x(y) = \frac12 (\Phi(x-y) + \Phi(-x-y))$, and for $x,z\in\R$ the set
$M_{x,z} = \{y\in\R\colon \Psi_x(y) - \Psi_z(y)>0\}$.
If for all $F\in\F$ and for all $x,z\in\R$ with $|x|>|z|$ one has that $\P(Y - C(F)\in M_{x,z})>0$, $Y\sim F$, then $S$ is strictly metrically $\F$-order-sensitive for $C$. In particular, if for all $F\in\F$ and for all $x\neq0$ it holds that $\P(Y - C(F)\in M_{x,0})>0$, $Y\sim F$, then $S$ is strictly $\F$-consistent for $C$.
\end{prop}

\begin{proof}
Let $|x|>|z|$. Note that due to the convexity of $\Phi$, it holds that $\Psi_x\ge \Psi_z$. Let $F\in\F$ with center of symmetry $c=C(F)$ and let $Y\sim F$. Then, using the fact that $\Phi$ is even and that $Y-c \stackrel{d}= c-Y$, one obtains
\begin{align*}
\bar S(c+x,F) - \bar S(c+z,F) 
&= \E_F[\Phi(x-(Y-c)) - \Phi(z-(Y-c))] \\
&= \E_F[\Psi_x(Y-c) - \Psi_z(Y-c)]>0\,.
\end{align*}
This shows the strict metrical $\F$-order-sensitivity. The strict $\F$-consistency follows upon taking $z=0$.
\end{proof}
If $\Phi$ is strictly convex then $M_{x,z}=\R$ for all $|x|>|z|$.

\begin{exmp}\label{exmp: Huber loss}
Let $\F$ be a class of symmetric distributions and $S(x,y) = \Phi(x-y)$. 
\begin{enumerate}[(i)]
\item
If $\Phi(t) = |t|^2$, the squared loss arises. Since $\Phi$ is strictly convex, the squared loss is strictly metrically $\F$-order-sensitive.
\item
For $\Phi(t) = |t|$, $S$ takes the form of the absolute loss. Then $S$ is strictly metrically $\F$-order-sensitive (and strictly $\F$-consistent) if and only if $C(F)\in\supp(F)$ for all $F\in\F$.\footnote{With the support of $F$ $\supp(F)$ we denote the support of the measure induced by $F$. In this context, $C(F)\in\supp(F)$ is equivalent to $F$ having a unique median.}
\item
Another prominent example of a metrically order-sensitive scoring function for the center of a symmetric distribution besides the absolute or the squared loss is the so-called \emph{Huber loss} which was presented in \cite{Huber1964} and arises upon taking $S(x,y) = \Phi(x-y)$ with
\[
\Phi(t) = \begin{cases}
\tfrac12 t^2, & \text{for } |t|<k, \\
k|t| - \tfrac12 k^2, & \text{for }|t|\ge k,
\end{cases}
\]
where $k\in\R$, $k\ge0$ is a tuning parameter. The Huber loss is strictly metrically $\F$-order-sensitive (strictly $\F$-consistent) if and only if $[C(F) - k, C(F) + k]\cap \supp(F)\neq\emptyset$ for all $F\in\F$.
\end{enumerate}
\end{exmp}


We emphasize that there are not only metrically-order sensitive strictly consistent scoring functions for the center of symmetric distributions. One can also use asymmetric scoring functions, for example those for the median or the mean, to elicit the center of symmetry.

Due to the negative result of Proposition \ref{prop:quantiles and metric order sensitivity} we dispense with an investigation of scoring functions that are metrically order-sensitive for vectors of different quantiles.

\subsubsection{Expectiles}

The special situation of the $1/2$-expectile, which coincides with the mean functional, was already considered in Subsection \ref{subsubsec:ratio of exp}, so let $\tau\neq1/2$. It is obvious that the canonical scoring function for the $\tau$-expectile, that is, the asymmetric squared loss $$S_\tau(x,y) = |\one\{y\le x\} - \tau|(x-y)^2$$ is not metrically order-sensitive since $x\mapsto S_\tau(x+y,y)$ is not an even function. 
\emph{A fortiori}, it turns out that (under some assumptions) there is no strictly metrically $\F$-order-sensitive scoring function for the $\tau$-expectile for $\tau\neq1/2$. 
\begin{prop}\label{prop:expectiles}
Let $\tau\in(0,1)$, $\tau\neq1/2$, and $T_\tau = \mu_\tau\colon\F\to\A\subseteq\R$, $\interior(\A)\neq\emptyset$ convex, be the $\tau$-expectile. Assume that $T_\tau$ is surjective, and that Assumption (V1) in \cite{FisslerZiegel2016} holds with respect to the strict $\F$-identification function $V_\tau(x,y) = 2|\one\{y\le x\} - \tau|\,(x-y)$. Suppose that $\bar V(\cdot, F)$ is twice differentiable for all $F\in\F$ and that there is a strictly $\F$-consistent scoring function $S\colon\A\times\R\to\R$ such that $\bar S(\cdot, F)$ is three times differentiable for all $F\in\F$. In particular, let each $F\in\F$ be differentiable with derivative $f = F'$.\\
If there is a $t\in\A$ and $F_1,F_2\in\F$ such that $T_\tau(F_1) = T_\tau (F_2) = t$, $F_1(t) = F_2(t)$, but $F'_1(t) = f_1(t)\neq f_2(t) = F'_2(t)$, then $S$ is not metrically $\F$-order-sensitive.
\end{prop}

\begin{proof}
Under the assumptions, Osband's principle yields the existence of a function $h\colon\interior(\A)\to\R$, $h>0$ (by an argument like in the proof of Proposition \ref{prop:unique}) such that for all $F\in\F$, $x\in\interior(\A)$
\[
\frac{\diff}{\diff x} \bar S(x,F) = h(x) \bar V(x,F).
\]
Using the same argument as in the proof of Osband's principle \cite[Theorem 3.2]{FisslerZiegel2016}, $h$ is twice differentiable. Assume that $S$ is metrically $\F$-order sensi{\-}tive. Then, due to Lemma \ref{lem:metrical vs symmetry2},
for any $F\in\F$ the function $g_F\colon \A\ni x\mapsto g_F(x) = \bar S(T_\tau(F)+x,F)$ is an even function. Hence, invoking the smoothness assumptions, the third derivative of $g_F$ must be odd. So necessarily $g_F'''(0)=0$. Denoting $t_F = T_\tau(F)$, some tedious calculations lead to
\be{eq:g_F}
g_F'''(0) = 2h'(t_F)\big(F(t_F) (1-2\tau) +\tau\big) + 2h(t_F) f(t_F) (1-2\tau).
\ee
Recalling that $h>0$ and $\tau\neq1/2$ implies $g_{F_1}'''(0)\neq g_{F_2}'''(0)$. So $S$ cannot be metrically $\F$-order-sensitive.
\end{proof}
%

Inspecting the proof of Proposition \ref{prop:expectiles}, equation \eqref{eq:g_F} yields for $\tau=1/2$
\[
g_F'''(0) = h'(t_F)
\]
for any $F\in\F$, $t_F = T_\tau(F)$. With the surjectivity of $T_\tau$ this proves that $h'=0$, such that $h$ is necessarily constant. Hence, we get an alternative proof that the squared loss is the only strictly metrically order-sensitive scoring function for the mean, up to equivalence.

\subsection{Order-sensitivity on line segments}

Recalling Lemma \ref{lem:implications}, every componentwise order-sensitive scoring function is also order-sensitive on line segments. 
However, for the particular class of \emph{linear functionals}, the following corollary shows that any strictly consistent scoring function is already strictly componentwise order-sensitive on line segments.\footnote{According to \cite{AbernethyFrongillo2012}, we call a functional $T\colon \F\to\A$ \emph{linear}, if it behaves linearly for mixtures of distributions. That is, for any $F,G\in\F$ such that $(1-\lambda)F+\lambda G\in\F$ for $\lambda\in[0,1]$ it holds that $T((1-\lambda)F+\lambda G) = (1-\lambda)T(F) + \lambda T(G)$. Examples of linear functionals are expectations of transformations, that is, $T(F) = \E_F[p(Y)]$ for some $\F$-integrable function $p\colon \O\to\R^k$.}

\begin{cor}\label{cor:linear functionals}
If $\F$ is convex and $T\colon \F\to\A\subseteq \R^k$ is linear and surjective, then any strictly $\F$-consistent scoring function for $T$ is strictly $\F$-order-sensitive on line segments.
\end{cor}

\begin{proof}
The linearity of $T$ implies that $T$ is mixture-continuous. Then the assertion follows directly by Proposition \ref{prop:decr paths} and the special form of the image of the path $\gamma$ in the proof therein, which is a line segment.
\end{proof}

Corollary \ref{cor:linear functionals} immediately leads the way to the result that the class of strictly order-sensitive scoring functions on line segments is strictly bigger than the class of strict componentwise order-sensitive scoring functions (for some functionals with dimension $k\ge2$.) E.g.~consider a vector of expectations satisfying the conditions of Proposition \ref{prop:sum} which are the same as the ones in \citet[Proposition 4.4]{FisslerZiegel2016}. Due to the latter result, there are strictly consistent scoring functions -- and hence, with Corollary \ref{cor:linear functionals}, strictly order-sensitive on line segments -- which are not additively separable. By Proposition \ref{prop:sum} they cannot be strictly componentwise order-sensitive.

We can extend the result of Corollary \ref{cor:linear functionals} to the case of ratios of expectations with the same denominator. 

\begin{lem}\label{lem:ratioexp}
Let $T\colon \F\to\A\subseteq\R^k$ be a ratio of expectations with the same denominator, that is, $T(F) = \bar p(F) / \bar q(F)$ for some $\F$-integrable functions $p\colon \O\to\R^k$, and $q\colon\O\to\R$ where we assume that $\bar q(F)>0$ for all $F\in\F$ and that $\A$ is open and convex. Any scoring function of the form
\begin{equation}\label{eq:ratioexp}
S(x,y) = -\phi(x)q(y) + \nabla \phi(x)(q(y)x - p(y))
\end{equation}
is strictly $\F$-order sensitive on line segments, where $\phi$ is strictly convex differentiable function on $\A$. 
\end{lem}
\begin{proof}
Let $F\in\F$, $t=T(F)$, $v\in\mathbb S^{k-1}$ and $0 \le s < s'$ such that $t+sv, t+s'v\in \A$. Then $\bar S(t + sv,F) = \bar q(F)(-\phi(t + sv) + s\nabla \phi(t+sv) v)$. The subgradient inequality yields
\begin{align*}
\bar S(t + sv,F) - \bar S(t + s'v,F) &< \bar{q}(F)\big((s'-s)\nabla \phi(t + s'v) v  \\&\qquad\qquad+ s\nabla \phi(t + sv)v - s'\nabla \phi(t + s'v)v\big) \le 0.
\end{align*} 
\end{proof}

\citet[Proposition 4.4]{FisslerZiegel2016} shows that essentially all strictly consistent scoring functions for $T$ in the above Lemma \ref{lem:ratioexp} are of the form at \eqref{eq:ratioexp}; see also \citet[Theorem 13]{FrongilloKash2014b}.

Order-sensitivity on line segments is stable under applying an isomorphism via the revelation principle \cite[Theorem 4]{Gneiting2011}. However, dropping the linearity assumption on the bijection in the revelation principle, order-sensitivity on line segments is generally not preserved; see Subsection \ref{subsec:mean, variance}.

\begin{lem}\label{lem:isomorphism}
Let $S\colon \A\times \O\to\R$ be a (strictly) $\F$-order-sensitive scoring function on line segments for a functional $T\colon \F\to\A\subseteq \R^k$. Let $g\colon\A\to\A'\subseteq \R^k$ be an isomorphism where $\A'$ is the image of $\A$ under $g$. Then $S_g\colon \A'\times \O\to\R$ defined as $S_g(x',y) = S(g^{-1}(x'),y)$ is a (strictly) $\F$-order-sensitive scoring function on line segments for the functional $T_g = g\circ T\colon \F\to\A'$.
\end{lem}

\begin{proof}
Let $F\in \F$, $t = T(F)$ and $t_g = T_g(F) = g(t)$. Let $v\in \mathbb S^{k-1} $ and $s\in[0,\infty)$. 
Using the linearity of $g^{-1}$ we get
\[
\bar S_g(t_g + sv,F) = \bar S\big(g^{-1}(g(t) + sv),F\big) = \bar S(t + sg^{-1}(v),F).
\]
Since also $g$ is an isomorphism, we have that $g^{-1}(v)/\|g^{-1}(v)\|\in \mathbb S^{k-1}$. Hence, the map $s\mapsto \bar S_g(t_g + sv,F)$ is (strictly) increasing for all $v\in \mathbb S^{k-1}$ if $S$ is (strictly) order-sensitive on line segments.
\end{proof}

\subsubsection{The pair (mean, variance)}\label{subsec:mean, variance}

The pair (mean, variance) is of importance not only from an applied point of view but it is also an interesting example in the theory about elicitability.
Due to the lack of convex level sets, variance is not elicitable \cite[Theorem 6]{Gneiting2011}. However, the pair (mean, variance) is a bijection of the (elicitable) pair (mean, second moment), and, invoking the revelation principle \cite[Theorem 4]{Gneiting2011}, variance is \emph{jointly} elicitable with the mean. The revelation principle provides an explicit link between the class of strictly consistent scoring functions for the first two moments which are of Bregman-type \cite[Proposition 4.4]{FisslerZiegel2016} and the respective class for mean and variance.

As the pair (mean, variance) has of a non-elicitable component, if fails to be component\-wise order-sensitive (Lemma \ref{lem:components}) and therefore, it is also not metrically order-sensitive. \textit{A priori}, order-sensitivity on line segments is not ruled out. 
Corollary \ref{cor:linear functionals} implies that any strictly consistent scoring function for the pair of the first and second moment is order-sensitive on line segments.
Even though the bijection connecting (mean, variance) with the pair of the first two moments is not linear, and, hence, we cannot apply Lemma \ref{lem:isomorphism}, the following proposition gives necessary and sufficient conditions for scoring functions to be order-sensitive on line segments for  (mean, variance). Example \ref{exmp:os for mean, variance} shows the existence of order-sensitive scoring functions on line segments for (mean, variance).

\begin{prop}\label{prop:mean variance os}
Let $\F$ be a class of distributions on $\R$ with finite second moments such that the functional $T = (\textup{mean, variance})\colon \F\to\A$ is surjective on $\A = \R\times (0,\infty)$.
Let Assumptions (F1) and (V1) from \cite{FisslerZiegel2016} be satisfied with the strict $\F$-identification function $V\colon\A\times \R\to\R^2$, $V(x_1,x_2,y) = \big(x_1 - y, x_2 + x_1^2 - y^2 \big)^\top$. 
Let $S\colon\A\times \R\to\R$ be a scoring function that is (jointly) continuous and for any $y\in\R$, the function $\A\ni x\mapsto S(x,y)$ be twice continuously differentiable. Then $S$ is $\F$-order-sensitive on line segments for $T$ if and only if $S$ is of the form
\be{eq:S mean variance}
S(x_1,x_2,y) = -\phi(x_1,x_2+x_1^2) + \nabla \phi(x_1,x_2+x_1^2)
\begin{pmatrix}
x_1 - y \\
x_2+x_1^2 - y^2
\end{pmatrix} + a(y),
\ee
where $a\colon\R\to\R$ is some $\F$-integrable function and $\phi\colon\A' \to\R$, $\A' = \{(x_1,x_2+x_1^2)\in\R^2\,|\,x\in\A\} = \{(m_1,m_2)\in\R^2\,|\,m_1^2<m_2\}$, is a convex, three times continuously differentiable function such that the second order partial derivatives $\phi_{ij}:=\partial_i \partial_j \phi$ satisfy 
\begin{align}
\phi_{12}(m_1,m_2) &= -2m_1\phi_{22}(m_1,m_2) \label{eq:condition eq}\\
\phi_{11}(m_1,m_2)&\ge (m_2 + 3m_1^2)\phi_{22}(m_1,m_2)
\label{eq:condition ineq}
\end{align}
for all $(m_1,m_2)\in\A'$.
\end{prop}

\begin{proof}
Let $S$ be $\F$-order-sensitive on line segments. This implies that $S$ is $\F$-consistent. 
Using the revelation principle, $S'\colon\A'\times\R\to\R$, 
\be{eq:revelation}
S'(m_1,m_2,y) = S(m_1,m_2-m_1^2,y)
\ee
is an $\F$-consistent scoring function for $T' = (T_1, T_2 + T_1^2)\colon\F\to\A'$, the pair of the first and second moment.
Moreover, $S'$ fulfils the same regularity conditions as $S$. 
\citet[Proposition 4.4]{FisslerZiegel2016} holds \textit{mutatis mutandis} also for consistent scoring functions with $\phi$ convex. It is straight forward to check that the conditions for \citet[Proposition 4.4]{FisslerZiegel2016} are fulfilled for $S'$ and $T'$ with the canonical identification function $V'\colon\A'\times\R\to\R^2$, $V'(m_1,m_2,y) = \big(m_1 - y, m_2 - y^2 \big)^\top$. Hence, $S'$ is necessarily of the form
\[
S'(m_1,m_2,y) = -\phi(m_1,m_2) + \nabla \phi(m_1,m_2)
\begin{pmatrix}
m_1 - y \\
m_2 - y^2
\end{pmatrix} + a(y),
\]
where $a\colon\R\to\R$ is some $\F$-integrable function and $\phi\colon\A'\to\R$ is a convex $C^3$-function with gradient $\nabla\phi$ (considered as a row vector) and Hessian $\nabla^2 \phi = (\phi_{ij})_{i,j = 1,2}$. 
In summary, \eqref{eq:revelation} yields the form at \eqref{eq:S mean variance}.\\
Now, we verify conditions \eqref{eq:condition eq} and \eqref{eq:condition ineq}. 
Let $F\in\F$, with $(t_1,t_2)=T(F)$. For $v\in\R^2$, $\|v\|=1$, $s\in \R$ with $\bar s:=t+sv\in\A$, it holds that 
\begin{align} \nonumber 
\frac{\diff}{\diff s}\bar S(t+sv,F)  
&=s\,(v_1,v_2 +2v_1\bar s_1)\nabla^2\phi\big(\bar s_1, \bar s_2 + \bar s_1^2\big)
\begin{pmatrix}
v_1 \\ v_2 +2v_1\bar s_1 -sv_1^2
\end{pmatrix} \\ \label{eq:term1}
&= s\,(v_1,v_2 +2v_1\bar s_1)\nabla^2\phi\big(\bar s_1, \bar s_2 + \bar s_1^2\big)
\begin{pmatrix}
v_1 \\ v_2 +2v_1\bar s_1
\end{pmatrix} \\\label{eq:term2}
&\quad-s^2 v_1^3 \left(\phi_{12}\big(\bar s_1, \bar s_2 + \bar s_1^2\big) +2\bar s_1\phi_{22}\big(\bar s_1, \bar s_2 + \bar s_1^2\big)\right) \\\label{eq:term3}
&\quad-s^2 v_1^2v_2\, \phi_{22}\big(\bar s_1, \bar s_2 + \bar s_1^2\big).
\end{align}
Since $(\phi_{ij})_{i,j=1,2}$ is positive semi-definite, the term at \eqref{eq:term1} is non-negative and the term at \eqref{eq:term3} has the sign of $-v_2$.
Consider $v = (1,0)^\top$. Due to the surjectivity of $T$ it holds that for all $c_1\in\R$, $c_2>0$, $s\in\R$ there exists a distribution $F^+ \in\F$ such that $T_1(F^+) +s = c_1$ and $T_2(F^+)=c_2$. Hence, 
\begin{multline}\label{eq:S1}
s^{-1}\frac{\diff}{\diff s}\bar S(t+sv,F^+) 
=( 1,2c_1)\nabla^2\phi\big(c_1,c_2 + c_1^2\big)
\begin{pmatrix}
1 \\  2c_1
\end{pmatrix} \\ 
 - s\left( \phi_{12} \big(c_1,c_2 + c_1^2\big)  + 2c_1\phi_{22}\big(c_1,c_2 + c_1^2\big)\right).
\end{multline}
Due to the $\F$-order-sensitivity of $S$, the term on the left-hand side of \eqref{eq:S1} is non-negative for all $s\in\R$. 
Since $|s|$ can be arbitrarily large, the term $\phi_{12} \big(c_1,c_2 + c_1^2\big)  + 2c_1\phi_{22}\big(c_1,c_2 + c_1^2\big)$ must vanish and we obtain \eqref{eq:condition eq}.

Finally, let $v$ be such that $v_1,v_2\neq0$ and w.l.o.g.\ $v_2 > 0$. Then 
\begin{align*}
&s^{-1}\frac{\diff}{\diff s}\bar S(\bar s,F)  
= (v_1,v_2 +2v_1\bar s_1)\nabla^2\phi\big(\bar s_1, \bar s_2 + \bar s_1^2\big)
\begin{pmatrix}
v_1 \\ v_2 +2v_1\bar s_1
\end{pmatrix}
-s v_1^2v_2\, \phi_{22}\big(\bar s_1, \bar s_2 + \bar s_1^2\big) \\
&= v_1^2 \phi_{11} \big(\bar s_1, \bar s_2 + \bar s_1^2\big) - sv_1^2v_2 \phi_{22}\big(\bar s_1, \bar s_2 + \bar s_1^2\big)\\
&+ (v_2 + 2v_1\bar s_1)\big[ 2v_1 \phi_{12}\big(\bar s_1, \bar s_2 + \bar s_1^2\big) +4v_1\bar s_1 \phi_{22}\big(\bar s_1, \bar s_2 + \bar s_1^2\big) + (v_2-2v_1\bar s_1)\phi_{22}\big(\bar s_1, \bar s_2 + \bar s_1^2\big)\big] \\
&=v_1^2 \Big(\phi_{11} \big(\bar s_1, \bar s_2 + \bar s_1^2\big) 
- (sv_2 + 4\bar s_1^2)\phi_{22}\big(\bar s_1, \bar s_2 + \bar s_1^2\big)  \Big) + v_2^2 \phi_{22}\big(\bar s_1, \bar s_2 + \bar s_1^2\big) .
\end{align*}
Due to the surjectivity of $T$ it holds that for all $c_1\in\R$, $c_2>0$, $s< c_2/v_2$ there exists a distribution $F^+ \in\F$ such that $T_1(F^+) +sv_1 = c_1$ and $T_2(F^+) + sv_2=c_2$. Consequently, one obtains the lower bound
\[
s^{-1}\frac{\diff}{\diff s}\bar S(\bar s,F^+)\ge v_1^2 \Big(\phi_{11} \big(c_1, c_2 + c_1^2\big) - (c_2 + 4c_1^2)\phi_{22}\big(c_1, c_2 + c_1^2\big)\Big) + v_2^2 \phi_{22}\big(c_1, c_2 + c_1^2\big)
\]
and this bound is asymptotically attained for $s \uparrow c_2/v_2$. As $v_2$ can be arbitrarily small, it is necessary and sufficient for order sensitivity on line segments that 
the map $\A\ni (c_1,c_2)\mapsto \phi_{11} \big(c_1, c_2 + c_1^2\big) - (c_2 + 4c_1^2)\phi_{22}\big(c_1, c_2 + c_1^2\big)$ is non-negative which is equivalent to \eqref{eq:condition ineq}. The reverse direction follows with analogous considerations.
\end{proof}
\
\begin{exmp}\label{exmp:os for mean, variance}
An example for a class of strictly convex $C^3$-function $\phi\colon\A'\to\R$ satisfying \eqref{eq:condition eq} and \eqref{eq:condition ineq} with equality is given by 
\[
\phi(m_1,m_2) = \big(m_2 - m_1^2\big)^{-1} +b_1m_1 + b_2m_2 +b_3, \qquad b_1,b_2,b_3\in\R.
\]
For the case $b_1 = b_2 = b_3 =0$, the resulting scoring function of the form at \eqref{eq:S mean variance} is 
\be{eq:S mean variance hom}
S(x_1,x_2,y) = x_2^{-2}\big(x_1^2 - 2x_2 - 2x_1y + y^2\big).
\ee
Interestingly, this results not only in an order-sensitive scoring function on line segments for the pair (mean, variance), but it is also a mixed positively homogeneous scoring function of degree $-2$; see Section \ref{sec:homog}.
\end{exmp}

\subsubsection{The pair (Value at Risk, Expected Shortfall)}

Value at Risk (VaR) and Expected Shortfall (ES) are popular risk measures in banking and insurance. 
For a financial position $Y$ with distribution $F$ and a level $\alpha \in(0,1)$, they are defined as 
\begin{align*}
\VaR_\alpha(F) &:= F^{-1}(\alpha) = \inf\{x \in \mathbb{R} : F(x) \ge \alpha\}, \\
\ES_\alpha(F) &:= \frac{1}{\alpha}\int_0^\alpha \VaR_\beta(F) \,\mathrm d \beta \\
&= \frac{1}{\a}\E_F[Y\,\one\{Y\le \VaR_\alpha(F)\}]  +\frac{1}{\a} \VaR_\a(F)
\big(\a- F(\VaR_\a(F))\big)\,.
\end{align*}
Note that if $F$ is continuous at $\VaR_\a(F)$, that means, if $F(\VaR_\a(F))=\a$, one can write $\ES_\a(F)$ equivalently as $\E_F[Y\,|\,Y\le \VaR_\a(F)]$.
Our sign convention implies that risky positions yield large \emph{negative} values of $\text{VaR}_\alpha$ or $\text{ES}_\alpha$. Intuitively, $\text{VaR}_\alpha$ gives the worst loss out of the best $(1-\alpha)\times 100\%$ of all cases, whereas $\text{ES}_\alpha$ gives the average loss given one exceeds $\text{VaR}_\alpha$. Merits and pitfalls of these two important risk measures are discussed in \citet{EmbrechtsPuccettiETAL2014,EmbrechtsHofert2014} where numerous further references are given.

$\text{VaR}_\alpha$, as a quantile, is elicitable under mild regularity conditions, whereas $\text{ES}_\alpha$ fails to be elicitable \citep{Gneiting2011}. However, recently it was shown in \citet[Theorem 5.2 and Corollary 5.5]{FisslerZiegel2016} that the pair $(\VaR_\a, \ES_\a)$ is elicitable and the class of strictly convex scoring functions was characterized to be of the form \eqref{eq:S_VaR,ES} (under the conditions of Osband's principle, \citet[Theorem 3.2, Corollary 3.3]{FisslerZiegel2016}).
Note that the proof of \citet[Theorem 5.2(ii) and Corollary 5.5]{FisslerZiegel2016} is imprecise for the case that a distribution $F \in \mathcal{F}$ is not continuous at its $\alpha$-quantile. However, the arguments are easily adapted and the result holds as stated.

\begin{prop}\label{prop:os line segments VaR, ES}
Let $\a\in(0,1)$, $\F$ be a class of continuously differentiable distribution functions on $\R$ with finite first moments and unique $\a$-quantiles. Let $\mathsf{A}\subseteq \{(x_1,x_2)\in\mathbb R^2\colon x_1\ge x_2\}$ be convex. Define $\mathsf A_2$ as the projection of $\mathsf A$ onto the second coordinate axis and let $S:\mathsf A \times \R \to \R$ be a scoring function of the form 
\begin{align}\label{eq:S_VaR,ES}
S(x_1,x_2,y) &= \big(\one\{y\le x_1\} - \alpha\big)g(x_1) - \one\{y\le x_1\}g(y)\\ \nonumber
&+ \phi'(x_2)\Big(x_2 +\big(\one\{y\le x_1\} - \alpha\big)\frac{x_1}{\alpha} - \one\{y\le x_1\}\frac{y}{\alpha} \Big) 
- \phi(x_2),
\end{align}
with $g:\R \to \R$ differentiable and increasing and $\phi:\mathsf A_2 \to \R$ twice differentiable, and $\phi'>0, \phi''>0$. If
\begin{equation}\label{eq:sufficient2 prop}
\phi'(x) + (x-z)\phi''(x) \ge 0, \quad \text{for all $x,z \in \mathsf{A}_2$,}
\end{equation}
then $S$ is strictly $\F$-order-sensitive on line segments for $(\VaR_\a, \ES_\a)$.
\end{prop}

\begin{proof}
Let $F \in \mathcal{F}$ with density $f$, $t = (t_1,t_2) = T(F)$, $v = (v_1,v_2) \in \mathbb{S}^2$, and $s > 0$ such that $t + sv \in \mathsf A$. Then, after some calculation, we find
\begin{multline}\label{eq:ddsS}
\frac{\diff}{\diff s}\bar S(t_1 + sv_1,t_2 + sv_2, F) = (F(t_1 + sv_1) - \a)v_1\big(g'(t_1 + sv_1)+\frac{1}{\alpha}\phi'(t_2 + sv_2)\big) \\+ sv_2^2\phi''(t_2 + sv_2) + v_2\phi''(t_2 + sv_2)\Big(\frac{1}{\a}\int_{t_1}^{t_1 + sv_1}F(y)\diff y - sv_1\Big)\,.
\end{multline}
We have
\begin{equation}\label{eq:thirdsum}
\a s v_1 \le \int_{t_1}^{t_1 + sv_1}F(y)\diff y \le F(t_1 + sv_1) s v_1.
\end{equation}
Note that by assumption $g'\ge 0$, $\phi' > 0$, $\phi'' > 0$, and, furthermore $(F(t_1 + sv_1) - \a)v_1 > 0$ for $v_1 \not=0 $. Therefore, if $v_2 \ge 0$, the first two summands on the right-hand side of \eqref{eq:ddsS} are strictly positive and the last one is non-negative using the first inequality in \eqref{eq:thirdsum}. For $v_2 < 0$, we find using the second inequality in \eqref{eq:thirdsum}
\begin{align*}
\frac{\diff}{\diff s}\bar S(t_1 + sv_1,t_2 + sv_2, F) &> (F(t_1 + sv_1) - \a)v_1\frac{1}{\alpha}\phi'(t_2 + sv_2) \\ &\quad+ v_2\frac{1}{\a}\phi''(t_2 + sv_2)\big(F(t_1 + sv_1)  - \a \big)s v_1\\
&= \frac{1}{\a}(F(t_1 + sv_1) - \a)v_1 \big(\phi'(t_2 + sv_2) + sv_2\phi''(t_2 + sv_2)\big) \ge 0,
\end{align*}
where the last inequality is due to the assumption at \eqref{eq:sufficient2 prop}.
\end{proof}

\begin{exmp}\label{exmp:FZ order-sensitive}
Consider the action domain $\A = \{x\in\R^2\colon x_1\ge x_2, \ x_2<0\}$, so $\mathsf{A}_2 = (-\infty,0)$. For all $\phi$ in the family $\{\phi_b\colon (-\infty, 0)\to\R \colon b\in(0,1]\}$ where $\phi_1(x) = -\log(|x|)$, $x<0$ and for $b \in (0,1)$
\[
\phi_b(x) = \frac{1}{b-1}|x|^{1-b},  \quad x<0,
\]
condition \eqref{eq:sufficient2 prop} is satisfied. 
\end{exmp}

A strict $\F$-identification function $V\colon\A\times \O\to\R$ for a functional $T\colon\F\to\A\subseteq \R$ is oriented for $T$ if 
\be{eq:orientation k=1 b}
\bar V(x,F)>0\quad \Longleftrightarrow \quad x>T(F)
\ee
for all $F\in\F$, $x\in\A$ \citep{Lambert2008,SteinwartPasinETAL2014}. One possible generalization of orientation for higher-dimensional functionals is the following. 
Let $T\colon \F\to \A\subseteq\R^k$ be a functional with a strict $\F$-identification function $V\colon \A\times \O\to\R^k$. Then $V$ is called an \textit{oriented} strict $\F$-identification function for $T$ if
\[
 v^\top \bar V(T(F) + sv,F) > 0  \quad \Longleftrightarrow \quad s>0
\]
for all $v\in\mathbb S^{k-1} := \{x\in\R^k \colon \|x\|=1\}$, for all $F\in \F$ and for all $s\in \R$ such that $T(F) + sv\in\A$.

Our notion of orientation differs from the one proposed by \cite{FrongilloKash2014b}. In contrast to their definition, our definition is \emph{per se} independent of a (possibly non-existing) strictly consistent scoring function for $T$. Moreover, whereas their definition has connections to the convexity of the expected score, our definition shows strong ties to order-sensitivity on line segments.

If the gradient of an expected score induces an oriented identification function, then the scoring function is strictly order-sensitive on line segments, and vice versa. However, the existence of an oriented identification function is not sufficient for the existence of a strictly order-sensitive scoring function on line segments. The reason is that -- due to integrability conditions -- the identification function is not necessarily the gradient of some (scoring) function. 
%
%

%

\section{Equivariant functionals and order-preserving scoring functions}\label{sec:equivariance}

Many statistical functionals have an invariance or equivariance property. For example, the mean is a linear functional, and hence, it is equivariant under linear transformations. So $\E[\ph(X)] = \ph(\E[X])$ for any random variable $X$ and any linear map $\ph\colon \R\to\R$ (of course, the same is true for the higher-dimensional setting). On the other hand, the variance is invariant under translations, that is $\Var(X-c) = \Var(X)$ for any $c\in\R$, but scales quadratically, so $\Var(\lambda X) = \lambda^2 \Var(X)$ for any $\lambda \in\R$. The next definition strives to formalize such notions.

\begin{defn}[$\pi$-equivariance]\label{defn:generalized equivariance}
Let $\F$ be a class of probability distributions on $\O$ and $\A$ be an action domain. Let $\Phi$ be a group of bijective transformations $\ph\colon\O\to\O$, $\Phi^*$ a group of bijective transformations $\ph^*\colon\A\to\A$, and $\pi \colon \Phi \to \Phi^*$ be a map. A functional $T\colon\F\to\A$ is \emph{$\pi$-equivariant} if
for all $\ph\in\Phi$ 
\[
T(\L(\ph(Y))) = (\pi\ph)(T(\L(Y)))
\]
for all random variables $Y$ such that $\L(Y)\in \F$. 
\end{defn}

\begin{exmp}\label{exmp:equivariance}
\begin{enumerate}[(i)]
\item
For $\A=\O=\R$, the mean functional is $\pi$-equivariant for $\Phi = \Phi^* = \{x\mapsto x+c, c\in\R\}$ the translation group and $\pi$ the identity map, or for $\Phi = \Phi^* = \{x\mapsto \lambda x, \lambda \in\R\setminus \{0\}\}$ the multiplicative group and again $\pi$ the identity map.
\item
For $\A=\O=\R^k$, the multivariate mean functional is $\pi$-equivariant for $\Phi = \Phi^* = \{x\mapsto x+c, c\in\R^k\}$ the translation group and $\pi$ the identity map.
\item
For $\A=\O=\R$, Value at Risk at level $\a$, Expected Shortfall at level $\a$ and the $\tau$-expectile are $\pi$-equivariant for $\Phi = \Phi^* = \{x\mapsto x+c, c\in\R\}$ the translation group and $\pi$ the identity map, or for $\Phi = \Phi^* = \{x\mapsto \lambda x, \lambda >0\}$ the multiplicative group and again $\pi$ the identity map.
\item
For $\A = [0,\infty)$ and $\O = \R$, the variance is $\pi$-equivariant for $\Phi = \{x\mapsto x+c, c\in\R\}$ the translation group and $\Phi^* = \{\mathrm{id}_\A\}$ the trivial group consisting only of the identity on $\A$, such that $\pi$ is the constant map.
\item
For $\A = [0,\infty)$ and $\O = \R$, the variance is $\pi$-equivariant for $\Phi = \Phi^* = \{x\mapsto \lambda x,\lambda \in\R\setminus \{0\}\}$ the multiplicative group, and $\pi((x\mapsto \lambda x)) = (x\mapsto \lambda^2 x)$.
\item
Let $\A= \R^k$, $\O = \R$ and $T$ be the functional whose $m$th component is the $m$th moment. Then $T$ is $\pi$-equivariant with $\Phi = \{y\mapsto \lambda y, \lambda \in\R\setminus \{0\}\}$, $\Phi^* = \{x\mapsto (\lambda^m x_m)_{m=1}^k, \lambda \in\R\setminus \{0\}\}$, and $\pi((y\mapsto \lambda y)) = (x\mapsto (\lambda^m x_m)_{m=1}^k)$.
\end{enumerate}
\end{exmp}

If a functional $T$ is elicitable, $\pi$-equivariance can also be expressed in terms of strictly consistent scoring functions; see also \citet[p.~750]{Gneiting2011}.

\begin{lem}\label{lem:equivariance1}
Let $S\colon\A\times \O\to\R$ be a strictly $\F$-consistent scoring function for a functional $T\colon\F\to\A$ and let $\pi\colon \Phi \to\Phi^*$. Then, $T$ is $\pi$-equivariant if and only if for all $\ph\in\Phi$
\[
\argmin_{x\in\A} \bar S((\pi\ph)(x),\L(\ph(Y))) =  \argmin_{x\in\A} \bar S(x,\L(Y))
\]
for all random variables $Y$ such that $\L(Y)\in \F$. 
\end{lem}
The proof of Lemma \ref{lem:equivariance1} is direct.
It implies that the scoring function 
\be{eq:S_ph}
S_{\pi, \ph}\colon\A\times \O\to\R, \quad (x,y)\mapsto S_{\pi, \ph}(x,y) = S((\pi\ph)(x),\ph(y))
\ee
is also strictly $\F$-consistent for $T$. 
Similarly to the motivation of order-sensitivity of scoring functions, for fixed $\pi\colon \Phi\to\Phi^*$, it is a natural requirement on a scoring function $S$ that for all $\ph\in\Phi$ the ranking of \emph{any two} forecasts is the same in terms of $S$ and in terms of $S_{\pi, \ph}$. 

\begin{defn}[$\pi$-order-preserving]
Let $\pi\colon\Phi\to\Phi^*$. A scoring function $S\colon\A\times \O\to\R$ is \emph{$\pi$-order-preserving} with respect to $\F$ if for all $\ph\in\Phi$ one has
\[
\sgn\big(\bar S(x,F) - \bar S(x',F)\big) = \sgn\big(\bar S_{\pi, \ph}(x,F) - \bar S_{\pi, \ph}(x',F)\big)
\]
for all $F\in\F$ and for all $x,x'\in\A$, where $S_{\pi, \ph}$ is defined at \eqref{eq:S_ph}. $S$ is \emph{linearly $\pi$-order-preserving} if for all $\ph\in\Phi$ and for all $x,x'\in\A$ there is a $\lambda>0$ such that
\be{eq:defn:equivariance}
\lambda \big( S(x,y) - S(x',y)\big) = S_{\pi, \ph}(x,y) - S_{\pi, \ph}(x',y)
\ee
for all $y\in\O$. If $S$ is linearly $\pi$-order-preserving with a $\lambda>0$ independent of $x,x'\in\A$, then we call $S$ \emph{uniformly} linearly $\pi$-order-preserving.
\end{defn}
The following lemma is immediate.

\begin{lem}\label{lem:equivariance}
Let $\pi\colon\Phi\to\Phi^*$. If a scoring function $S\colon\A\times \O\to\R$ is linearly $\pi$-order-preserving, it is $\pi$-order-preserving with respect to any class $\F$ of probability distributions on $\O$.
\end{lem}
The two practically most relevant examples of uniform linear $\pi$-order preservingness are translation invariance and positive homogeneity of scoring functions, or, to be more precise, of score differences. They are described in the two subsequent subsections.

\subsection{Translation invariance}
Consider a translation equivariant functional such as the mean treated in Example \ref{exmp:equivariance} (ii). Then, a scoring function $S\colon \R^k\times \R^k\to\R$ is said to have translation invariant score differences if it is uniformly linearly $\pi$-equivariant with $\lambda = 1$ for all $\ph \in\Phi$. In formulae, we require $S$ to satisfy 
\begin{equation}\label{eq:tmp}
S(x-z,y-z) - S(x'-z,y-z) = S(x,y) - S(x',y)
\end{equation}
for all $x,x', y,z \in\R^k$. Note that what is particularly appealing is that the action domain and the observation domain coincide and, in particular, have the same dimension. However, there are also other functionals such as vectors of different quantiles or expectiles, or the vector $(\VaR_\a, \ES_\a)$ satisfying properties one can naturally call translation equivariant, but that have the drawback that $\A\neq \O$ (typically, $\O$ is of lower dimension than $\A$). Then, translation invariance means that the score is invariant under a simultaneous translation of the observation and the forecast along respective linear subspaces of $\A$ and $\O$. 

Let $\A \subseteq \R^k$, $\O= \R^d$ and $m\in \{1, \ldots, \min\{k,d\}\}$. Let $M_\O\in\R^{d\times m}$ and $M_\A\in\R^{k\times m}$ be two matrices with rank $m$.
Define the transformation groups 
\begin{align*}
\Phi &:= \Phi_{M_\O} :=\{y \mapsto y-M_\O z, \ z\in\R^m\},\\
\Phi^* &:= \Phi^*_{M_\A} :=\{x \mapsto x-M_\A z, \ z\in\R^m\},
\end{align*}
where we impose that $x-M_\A z \in \A$ for all $x \in \A$, $z \in \R^m$.
Then, the map $\pi = \pi_{M_\O, M_\A}\colon \Phi_{M_\O} \to \Phi^*_{M_\A}$ naturally induced by $M_\O$ and $M_\A$ is given as
\[
 \pi_{M_\O, M_\A}((y\mapsto y - M_\O z)) = (x \mapsto x- M_\A z).
\]
We say that a functional $T\colon\F\to\R^k$ is linearly equivariant if there are such matrices $M_\O, M_\A$ such that $T$ is $ \pi_{M_\O, M_\A}$-equivariant.

\begin{exmp}\label{exmp: equivariance 2}
\begin{enumerate}[(i)]
\item
Let $\O=\R$, $\A=\{(x_1,x_2)\in\R^2\colon  x_2\le x_1 \}$ and $T = (\VaR_\a, \ES_\a)$ with some generic $\F$. Then $T$ is $\pi_{M_\O, M_\A}$-equivariant with $M_\O = \mathrm{id}_\R$ and $M_{\A} = (1,1)^{\top}$.
\item
Let $\O=\R$, $\A = \R\times [0,\infty)$ and $T = (\textup{mean, variance})\colon \F\to\A$ where all $F\in\F$ have finite second moments. Then $T$ is $\pi_{M_\O, M_\A}$-equivariant with $M_\O = \mathrm{id}_\R$ and $M_{\A} = (1,0)^{\top}$.
\item
Let $\O = \R$, $\A=\R^k$ and $T$ be a vector of $k$ different quantiles. Let $M \in\R^{k\times k}$ have rank at least 1 and consider the functional $T_M = M(T)$. Then $T_M$ is $\pi_{M_\O, M_\A}$-equivariant with $M_\O = \mathrm{id}_\R$ and $M_{\A} = M(1,\dots,1)^{\top}$.
\end{enumerate}
\end{exmp}

Adopting this notion, we say that a scoring function $S\colon \A\times \R^d\to\R$ is linearly $(M_\O, M_\A)$-invariant for two matrices $M_\O\in\R^{d\times m}$, $M_\A\in\R^{k\times m}$ with $\rank(M_\O) = \rank(M_\A) = m\in \{1, \ldots, \min\{k,d\}\}$ if 
\[
S(x - M_\A z, y - M_\O z) = S(x,y)
\]
for all $x\in\A$, $y\in\R^d$, $z\in\R^m$. Similarly, we will speak about linearly $(M_\O, M_\A)$-invariant identification functions and score differences.

%

Given a certain functional $T\colon\F\to\R^k$ and some $M_\O\in\R^{d\times m}$, $M_\A\in\R^{k\times m}$ with $\rank(M_\O) = \rank(M_\A) = m\in \{1, \ldots, \min\{k,d\}\}$, one can wonder about the class of strictly consistent scoring functions that are linearly $(M_\O, M_\A)$-invariant.
Clearly, with respect to Lemma \ref{lem:equivariance1} and Lemma \ref{lem:equivariance}, this class is empty if the functional $T$ is not $\pi_{M_\O, M_\A}$-equivariant. In the situation that $\A = \O=\R^k$ and $M_\O = M_\A = \mathrm{id}_{\R^k}$ the following proposition characterizes the gradients of linearly $(\mathrm{id}_{\R^k}, \mathrm{id}_{\R^k})$-invariant strictly consistent scoring function (if such scoring functions exist).

\begin{prop}\label{prop:translation invariance}
Let $T\colon\F\to\R^k$ be a surjective, identifiable functional with a linearly $(\mathrm{id}_{\R^k}, \mathrm{id}_{\R^k})$-invariant strict $\F$-identification function $V\colon\R^k\times\R^k\to\R^k$. 
Then, the following assertions hold.
\begin{enumerate}[\rm (i)]
\item
$T$ is $\pi_{\mathrm{id}_{\R^k}, \mathrm{id}_{\R^k}}$-equivariant.
\item
Assume there is a strictly $\F$-consistent scoring function $S\colon \R^k\times\R^k\to\R$ for $T$ 
with linearly $(\mathrm{id}_{\R^k}, \mathrm{id}_{\R^k})$-invariant score differences. Then, under Assumptions (V1) and (S1) in \cite{FisslerZiegel2016}, there is a constant matrix $h\in\R^{k\times k}$ such that 
\be{eq:h constant}
\nabla \bar S(x,F) = h\,\bar V(x,F)
\ee
for all $x\in\R^k$ and for all $F\in\F$.
\end{enumerate}
\end{prop}
\begin{proof}
If a random variable $Y$ has distribution $F$ with $F\in\F$, we write $F-z$ for the distribution of $Y-z$ where $z\in\R^k$. To show the first part, 
consider any $F\in\F$ and $z\in\R^k$. Then
\[
0= \E_F[V(T(F),Y)] = \E_F[V(T(F) - z,Y-z)].
\]
Since $V$ is a strict $\F$-identification function for $T$, $T(F-z) = T(F) -z$.

For the second part, \citet[Theorem 3.2]{FisslerZiegel2016} implies that there exists a matrix-valued function $h\colon \R^k\to\R^{k\times k}$ such that 
\[
\nabla \bar S(x,F) = h(x)\bar V(x,F)
\]
for all $x\in\R^k$ and for all $F\in\F$. We will show that $h$ is constant. Since $\bar S(x,F) - \bar S(x',F) = \bar S(x-z, F-z) - \bar S(x'-z, F-z)$ for all $x, x',z\in\R^k$ and $F\in\F$, we obtain by taking the gradient with respect to $x$
\be{eq:kernel}
h(x)\bar V(x,F) = h(x-z) \bar V(x-z, F-z) =  h(x-z)\bar V(x,F),
\ee
where the second identity is due to the linear $(\mathrm{id}_{\R^k}, \mathrm{id}_{\R^k})$-invariance of $V$. So \eqref{eq:kernel} is equivalent to 
\[
\bar V(x,F) \in \ker\big(h(x-z) -h(x)\big).
\]
Now, one can use Assumption (V1) and \citet[Remark 3.1]{FisslerZiegel2016}, which implies that
\[
\ker\big(h(x-z) -h(x)\big) = \R^k.
\]
Since $x,z \in \R^k$ were arbitrary, the function $h$ is constant.
\end{proof}

Using \citet[Proposition 4.4]{FisslerZiegel2016} one can establish the converse of Proposition \ref{prop:translation invariance}: If $V$ is a linearly $(\mathrm{id}_{\R^k}, \mathrm{id}_{\R^k})$-invariant strict $\F$-identification function, then \eqref{eq:h constant} implies that $S$ has linearly $(\mathrm{id}_{\R^k}, \mathrm{id}_{\R^k})$-invariant score differences.
The following lemma shows how to normalize scores with translation invariant score differences to obtain a translation invariant score.

\begin{lem}\label{lem:normalisation}
Let $S\colon \R^k\times \R^k\to\R$ a strictly $\F$-consistent scoring function for $T\colon \F\to\R^k$ with linearly $(\mathrm{id}_{\R^k}, \mathrm{id}_{\R^k})$-invariant score differences. If for all $y\in\R^k$, the point measures $\delta_y$ are in $\F$ and the function $y\mapsto S(T(\delta_y),y)$ is $\F$-integrable, then 
\be{eq:S_0}
S_0(x,y) := S(x,y) - S(T(\delta_y),y)
\ee
is a linearly $(\mathrm{id}_{\R^k}, \mathrm{id}_{\R^k})$-invariant, non-negative, strictly $\F$-consistent scoring function for $T$.
\end{lem}

\begin{proof}
If $S$ has linearly $(\mathrm{id}_{\R^k}, \mathrm{id}_{\R^k})$-invariant score differences, $S$ satisfies \eqref{eq:tmp} for all $x,x',y,z\in\R^k$. Due to Lemma \ref{lem:equivariance1}, $T$ must be $\pi_{\mathrm{id}_{\R^k}, \mathrm{id}_{\R^k}}$-equivariant, hence, $T(\delta_y) - z = T(\delta_{y-z})$. This yields that $S_0$ defined at \eqref{eq:S_0} is linearly $(\mathrm{id}_{\R^k}, \mathrm{id}_{\R^k})$-invariant. 
Since $S$ and $S_0$ are of equivalent form, also $S_0$ is strictly $\F$-consistent for $T$. The non-negativity follows directly from the fact that $\F$ contains all point measures and from the strict consistency.
\end{proof}

In case of the mean functional on $\R$, Proposition \ref{prop:translation invariance} has already been shown by \cite{Savage1971} who showed that the squared loss is the only strictly consistent scoring function for the mean that is of prediction error form, up to equivalence.\footnote{That means that the scoring function is a function in $x-y$ only.} Furthermore it implies that general $\tau$-expectiles and $\a$-quantiles have essentially one linearly $(\mathrm{id}_{\R}, \mathrm{id}_{\R})$-invariant strictly consistent scoring function only, namely the canonical choices $S_\tau(x,y) = |\one\{y\le x\} - \tau|(x-y)^2$ and $S_\a(x,y) = (\one\{y\le x\} - \a)(x-y)$. 

The uniqueness -- up to equivalence -- disappears for $k>1$. For example, for the the 2-dimensional mean functional, the previous results yield that any scoring function $S\colon\R^2\times\R^2\to\R$ of the form
\[
S(x,y) = \frac{h_{11}}{2}(x_1 - y_1)^2 + \frac{h_{22}}{2}(x_2 - y_2)^2 + h_{12}y_2(y_1 - x_1)+ h_{12}x_2(x_1-y_1)
\] 
is strictly consistent for the 2-dimensional mean functional and linearly $(\mathrm{id}_{\R^2}, \mathrm{id}_{\R^2})$-invariant, for any $h_{11}>0$ and $h_{11}h_{22} - h_{12}^2>0$.

Due to the additive separability of strictly consistent scoring functions for vectors consisting of different quantiles and expectiles \cite[Proposition 4.2]{FisslerZiegel2016}, strictly consistent scoring functions that are linearly $(\mathrm{id}_{\R}, \mathrm{id}_{\R^k})$-invariant for these vectors are not unique. However, the only flexibility in that class consists in choosing different weights for the respective summands of the scores.

The pair $(\textup{mean, variance})$ is a  $\pi_{M_\O, M_\A}$-equivariant functional with $M_\O$ and $M_\A$ as in Example \ref{exmp: equivariance 2}(ii). Curiously, it has a linearly $(M_\O, M_\A)$-invariant identification function $V(x_1,x_2,y) = \big(x_1 - y, x_2 - (x_1-y)^2\big)^\top$ but does not possess a strictly consistent linearly $(M_\O, M_\A)$-invariant scoring function.

\begin{prop}
Let $\F$ be a class of distributions on $\R$ with finite second moments such that the functional $T = (\textup{mean, variance})\colon \F\to\A$ is surjective on $\A = \R\times I$,where $I\subseteq [0,\infty)$ is an interval.
Let Assumptions (F1) and (V1) from \cite{FisslerZiegel2016} be satisfied with the strict $\F$-identification functions $V\colon\A\times \R\to\R^2$, $V(x_1,x_2,y) = \big(x_1 - y, x_2 -(x_1-y)^2 \big)^\top$ and $V^*\colon\A\times \R\to\R^2$, $V^*(x_1,x_2,y) = \big(x_1 - y, x_2 + x_1^2 - y^2 \big)^\top$. 
Let $S\colon\A\times \R\to\R$ be a $\F$-consistent scoring function for $T$ that is (jointly) continuous, and for any $y\in\R$, the function $\A\ni x\mapsto S(x,y)$ is twice continuously differentiable. 
If $S$ has linearly $(M_\O, M_\A)$-invariant score differences, then there is a $\lambda\ge0$ and an $\F$-integrable functional $a\colon\R\to\R$ such that  
\[
S(x_1,x_2,y)= \lambda(x_1-y)^2 +a(y).
\]
In particular, $S$ cannot be strictly $\F$-consistent for $T$.
\end{prop}

\begin{proof}
\citet[Theorem 3.2]{FisslerZiegel2016} asserts that there is a matrix-valued function $h\colon\interior(\A)\to\R^{2\times2}$ such that for all $(x_1,x_2)\in\interior(\A)$ and for all $F\in\F$ we have 
\be{eq:alternative1}
\nabla \bar S(x,F) = h(x_1,x_2)\bar V(x_1,x_2,F).
\ee
Due to the special form of $V$ and Assumption (F1), this equation holds also pointwise for all $y\in\R$. Moreover, the function $h$ is continuously differentiable. Assume that $S$ has linearly $(M_\O, M_\A)$-invariant score differences. This implies, combined with the previous result, that for all $(x_1,x_2)\in\interior(\A)$, and for all $y,z\in\R$
\begin{align*}
h(x_1,x_2)V(x_1,x_2,y)
&=
\nabla_x S(x_1,x_2,y) \\
&= \nabla_x S(x_1+z,x_2,y+z) \\
&= h(x_1+z,x_2)V(x_1+z,x_2,y+z) \\
&= h(x_1+z,x_2)V(x_1,x_2,y).
\end{align*}
An application of Assumption (V1) (similarly to the proof of Proposition \ref{prop:translation invariance}) yields that $h$ must be necessarily constant in its first argument.

On the other hand, arguing as in the proof of Proposition \ref{prop:mean variance os} with identification function $V^*$, the revelation principle yields that 
\be{eq:alternative}
S(x_1,x_2,y) = -\phi(x_1,x_2+x_1^2) + \nabla \phi(x_1,x_2+x_1^2)
\begin{pmatrix}
x_1 - y \\
x_2 + x_1^2 - y^2
\end{pmatrix} +a(y),
\ee
where $a\colon\R\to\R$ is some $\F$-integrable function and, due to our assumptions and \citet[Proposition 4.4]{FisslerZiegel2016}, $\phi\colon\A'\to\R$ is $C^3$ and convex with gradient $\nabla \phi$ and Hessian $(\phi_{ij})_{i,j=1,2}$. Using the representation at \eqref{eq:alternative}, one obtains
$\partial_2 S(x_1,x_2,y) = \phi_{22}(x_1,x_2+x_1^2)(x_2+x_1^2-y^2)$. A comparison to the form at \eqref{eq:alternative1} yields that
\begin{align*}
h_{22}(x_1,x_2) &= \phi_{22}(x_1,x_2+x_1^2) \\
h_{21}(x_1,x_2) &= 2x_1\phi_{22}(x_1,x_2+x_1^2).
\end{align*}
Since $\partial_1h_{22}(x_1,x_2)$ vanishes, we obtain that $0 = \partial_1 h_{21}(x_1,x_2) = 2\phi_{22}(x_1,x_2+x_1^2)$. As the Hessian of $\phi$ must be positive semi-definite $\phi_{11}\ge0$ and $\phi_{12} = \phi_{21}=0$. Since $\phi$ is $C^3$, we have that $\partial_2\phi_{11} = \partial_1\phi_{12} = 0$, hence $\phi_{11}$ is constant in the first argument. Equating the first component of \eqref{eq:alternative1} and \eqref{eq:alternative} and using that $h_{12} = h_{21} = 0$, we find that $\phi_{11}(x_1,x_2 + x_1^2) = h_{11}(x_1,x_2)$. 
As $h_{11}$ is constant in $x_1$, this implies that $\phi_{11}$ is also constant in its second argument which yields the claim.
\end{proof}

The functional $(\VaR_\a, \ES_\a)$, $\a\in(0,1)$, is also a relevant $\pi_{M_\O, M_\A}$-equivariant functional with $M_\O$ and $M_\A$ as in Example \ref{exmp: equivariance 2}(i). However, scoring functions with linearly $(M_\O,M_\A)$-invariant score differences only exist for restricted classes of distribution functions $\F$ which may not be natural choices in risk management applications.


\begin{prop} \label{prop:VaR,ES translation invariance}
Let $\a\in(0,1)$. Let $\F$ be a class of distribution functions on $\R$ with finite first moments and unique $\a$-quantiles. Consider $T = (\VaR_\a, \ES_\a)\colon \F\to\{(x_1,x_2)\in\R^2\colon  x_2\le x_1 \}$.
Then, the following assertions hold:
\begin{enumerate}[\rm (i)]
\item
Suppose there is some $c>0$ such that 
\be{eq:condition translation inv}
\ES_\a(F) + c > \VaR_\a(F) \qquad \text{for all $F\in\F$.}
\ee
That is, $T(\F) \subseteq \A_c := \{(x_1,x_2)\in\R^2\colon  x_2 \le x_1< x_2 +c\}$.
Then, any scoring function $S\colon\A_c\times \R\to\R$, which is equivalent to
\begin{multline}\label{eq:S_c}
S_c(x_1,x_2,y) = (\one\{y\le x_1\} - \a)c(x_1-y)  + \a (x_2^2/2 + x_1^2/2 - x_1x_2) \\ 
+ \one\{y\le x_1\}(-x_2(y-x_1) + y^2/2 - x_1^2/2),
\end{multline}
is strictly $\F$-consistent for $T$ and has linearly $(M_\O,M_\A)$-invariant score differences with $M_\O = \mathrm{id}_\R$, $M_\A = (1,1)^{\top}$.
\item
Under the conditions of \citet[Theorem 5.2(iii)]{FisslerZiegel2016}, 
there are strictly $\F$-consistent scoring functions for $T$ with linearly $(M_\O,M_\A)$-invariant score differences if and only if there is some $c>0$ such that \eqref{eq:condition translation inv} holds. Then, any such scoring function is necessarily equivalent to $S_d$ defined at \eqref{eq:S_c} almost everywhere, with $d\ge c$.
\end{enumerate}
\end{prop}
\begin{proof}
The scoring function $S_c$ is of equivalent form as given at \eqref{eq:S_VaR,ES} with $g(x_1) = -x_1^2/2 + cx_1$ and $\phi(x) = (\a/2)x_2^2$. This means that $\phi$ is strictly convex and the function $x_1\mapsto x_1\phi'(x_2)/\a + g(x_1)$ is strictly increasing in $x_1$ if and only if $x_2 + c>x_1$, that is, if and only if $(x_1, x_2)\in\A_c$, such that we obtain the $\F$-consistency of $S_c$ with \citet[Theorem 5.2(ii)]{FisslerZiegel2016}. 
A direct computation yields that $S_c(x_1+z,x_2+z,y+z) = S_c(x_1,x_2,y)$ for all $(x_1,x_2)\in\A_c$, $y,z\in\R$. This proves the first part.

Under the conditions of \citet[Theorem 5.2(iii)]{FisslerZiegel2016}, any strictly $\F$-consistent scoring function $S\colon\A\times \R\to\R$, where $\A = T(\F)$, is almost everywhere of the form given at 
\eqref{eq:S_VaR,ES} with $g$ continuously differentiable and $\phi$ twice continuously differentiable.
By translation invariance of score differences the function $\Psi\colon \R\times \A \times \A \times \R \to\R$, 
\begin{align*}
\Psi(z,x_1,x_2, x_1', x_2', y) & = S(x_1 +z,x_2+z,y+z) - S(x'_1 +z,x'_2+z,y+z) \\
&\quad - S(x_1,x_2,y) + S(x'_1,x'_2,y)
\end{align*}
constantly vanishes. Let $z, y\in\R$ and $(x_1, x_2), (x_1', x_2')\in\A$. Then
\begin{multline*}
0 = \frac{\diff}{\diff x_2} \Psi(z,x_1,x_2, x_1', x_2', y) 
= \big(x_2 - x_1 + \frac{1}{\a}\one\{y\le x_1\}(x_1-y)\big)(\phi''(x_2+z) - \phi''(x_2)),
\end{multline*}
hence $\phi''$ is constant, that is,  $\phi(x_2) = d_1x_2^2 + d_2x_2 + d_3$ with $d_1>0$ (ensuring the strict convexity of $\phi$) and $d_2,d_3\in\R$. Similarly, the derivative of $\Psi$ with respect to $z$ must vanish for all $z, y\in\R$ and $(x_1, x_2)$, $(x_1', x_2')\in\A$. 
A calculation yields 
\begin{align*}
0 = \frac{\diff}{\diff z} \Psi(z,x_1,x_2, x_1', x_2', y) 
&= \big(\one\{y\le x_1\} - \alpha\big) g'(x_1 + z) - \one\{y\le x_1\} g'(y+z) \\
&\quad - \big(\one\{y\le x'_1\} - \alpha\big) g'(x'_1 + z) + \one\{y\le x'_1\} g'(y+z) \\
&\quad +\frac{2d_1}{\a} \big(  \one\{y\le x_1\}(x_1-y) - x_1\big) \\
&\quad -\frac{2d_1}{\a} \big( \one\{y\le x'_1\}(x'_1-y) - x'_1\big).
\end{align*}
This implies that necessarily $g'(x_1) = (-2d_1/\a)x_1 + d_4$ for some $d_4\in\R$. Hence, $g(x_1) = (-d_1/\a)x_1^2 + d_4x_1 + d_5$ for some $d_5\in\R$. Now, by \citet[Theorem 5.3(iii)]{FisslerZiegel2016}, the function
\[
\psi_{x_2}(x_1) = x_1\phi'(x_2)/\a + g(x_1)
= x_1(2d_1x_2 +d_2)/\a - d_1x_1^2/\a + d_4x_1 +d_5
\]
must be strictly increasing in $x_1$ which holds if and only if 
\[
x_2 + \frac{d_2+d_4\a}{2d_1}  >x_1.
\]
This condition is satisfied for all $(x_1,x_2)\in\A=T(\F)$ if and only there is a $c>0$ such that $T(\F)\subseteq \A_c$  and 
$d:= (d_2+d_4\a)/(2d_1)\ge c$.
The scoring function at \eqref{eq:S_VaR,ES} with $\phi(x_2) =d_1x_2^2 + d_2x_2 + d_3$,  $d_1>0$, $d_2, d_3\in\R$, $g(x_1)= (-d_1/\a)x_1^2 + d_4x_1 + d_5$, $d_4, d_5 \in\R$ is equivalent to $S_d$ defined at \eqref{eq:S_c}, which concludes the proof.
\end{proof}

The scoring function $S_c$ has a close relationship to the class of scoring functions $S^W$ proposed in \cite{AcerbiSzekely2014}; see \citet[Equation (5.6)]{FisslerZiegel2016}. Indeed, $S_c(x_1, x_2, y) = c\big(\one\{y\le x_1\} - \a\big) (x-y) + S^W(x_1,x_2,y)$ with $W=1$. That means it is the sum of the standard $\a$-pinball loss for $\VaR_\a$ -- which is translation invariant -- and $S^1$. In the same flavor, the condition at \eqref{eq:condition translation inv} is similar to the one at \citet[Equation (5.7)]{FisslerZiegel2016}. Since $\ES_\a\le \VaR_\a$, the maximal action domain where $S_c$ is strictly consistent is the stripe $\A_c=\{(x_1,x_2)\in\R^2\colon  x_2 \le x_1< x_2 +c\}$. Of course, by letting $c\to\infty$, one obtains the maximal sensible action domain $\{(x_1,x_2)\in\R^2\colon x_1\ge x_2\}$ for the pair $(\VaR_\a, \ES_\a)$. However, considering the properly normalized version $S_c/c$, this converges to a strictly consistent scoring function for $\VaR_\a$ as $c\to\infty$, but which is independent of the forecast for $\ES_\a$. Hence, there is a caveat concerning the tradeoff between the size of the action domain and the sensitivity in the ES-forecast. This might cast doubt on the usage of scoring functions with translation invariant score differences for $(\VaR_\a, \ES_\a)$ in general. 

Interestingly, the scoring function $S_c$ at \eqref{eq:S_c} has positively homogeneous score differences if and only if $c =0$. However, $\A_0 = \emptyset$, which means that the requirement of translation invariance and homogeneity for score differences are mutually exclusive in case of strictly consistent scoring functions for $(\VaR_\a, \ES_\a)$.

\subsection{Homogeneity}\label{sec:homog}

If one is interested in a positively homogeneous functional of degree one such as the mean, expectiles, quantiles, or ES, a scoring function $S\colon\R\times\R\to\R$ is said to have positively homogeneous score differences of degree $b\in\R$ for this functional if the scoring function is uniformly linearly $\pi$-equivariant with $\Phi=\{\R\ni x\mapsto c x \in \R, c>0\}$ the multiplicative group, $\pi$ the identity on $\Phi$ and $\lambda = c^b$ in \eqref{eq:defn:equivariance}. This means that $S$ needs to satisfy 
\be{eq:homogeneity 1}
S(cx,cy) - S(cz,cy) = c^b\big( S(x,y) - S(z,y) \big)
\ee
for all $x,z,y\in\R$ and $c>0$. Since positive homogeneity of score differences is equivalent to invariance of forecast rankings under a change of unit, it has been argued that it is important in financial applications \citep{AcerbiSzekely2014}. \cite{NoldeZiegel2017} give a characterization of scoring functions with positively homogeneous score differences for many risk measures of applied interest, such as VaR\,/\,quantiles, expectiles, and the pair (VaR, ES); cf.\ \cite{Patton2011} for results concerning the mean functional.


If the functional $T$ is vector-valued, the degree of homogeneity can be different in the respective components, e.g.\ in case of the pair (mean, variance) or the vector consisting of the first $k$ moments; cf.\ Example \ref{exmp:equivariance}(vi). One can denote this property by \emph{mixed positive homogeneity}, which means in case of the vector of the first $k$ moments that
\be{}
T(\L(cY)) = \Lambda(c)T(\L(Y))
\ee
for all $c>0$, where $\Lambda(c)$ is the $k\times k$-diagonal matrix with diagonal elements $c,c^2, \ldots, c^k$.\footnote{Of course, we tacitly assume that for all $x\in\A$ and for all $c>0$, we have $\Lambda(c)x\in\A$.}
In this situation, an interesting instant for uniformly linearly $\pi$-order-preserving scoring functions $S\colon\A\times \R\to\R$ are those with \emph{mixed positively homogeneous} score differences of degree $b\in\R$. That is, 
\be{eq:mixed homogeneity 3}
S(\Lambda(c)x,cy) - S(\Lambda(c)z, cy) = c^b \big(S(xy) - S(z, y)\big)
\ee
for all $x,z\in\A$, $y\in\R$, and for all $c>0$. With $k=2$, corresponding assertions hold for the pair (mean, variance) and the respecitve scoring functions.

\begin{prop}\label{prop:mixed moments}
Let $\A\subseteq \R^k$ such that $\Lambda(c)x\in\A$ for all $c>0$, $x\in\A$. Let $S\colon\A\times\R\to\R$ be a consistent scoring function for the vector of the first $k$ moments of the form
\be{eq:score}
S(x,y) = -\phi(x) + \nabla\phi(x)
\left(x - (y,y^2, \ldots, y^k)^\top\right) + a(y),
\ee
where $\phi\colon\A\to\R$ is convex and differentiable with gradient $\nabla \phi$ (considered as a row vector).
Then $S$ has mixed positively homogeneous score differences of degree $b\in\R$ if and only if for all $c>0$ the map 
\be{eq:cond mixed}
x\mapsto \nabla\phi(\Lambda(c)x)\Lambda(c) -c^b \nabla\phi(x)
\ee
is constant.
\end{prop}

\begin{proof}
Suppose $\phi$ satisfies \eqref{eq:cond mixed}. 
This implies that for any $c>0$ the map $z\mapsto \phi(\Lambda(c)z) - c^b\phi(z)$ is an affine function. Moreover, a Taylor expansion yields that for all $x,z\in\A$
\[
 \phi(\Lambda(c)z) - c^b\phi(z)= \big(\nabla\phi(\Lambda(c)x)\Lambda(c) - c^b\nabla\phi(x)\big) (z-x)  + \phi(\Lambda(c)x) - c^b\phi(x).
\]
Then, a direct calculation yields the result.

Now, suppose \eqref{eq:mixed homogeneity 3} is satisfied. Its left-hand side equals 
\begin{gather*}
- \phi(\Lambda(c)x) + \nabla\phi(\Lambda(c)x)\Lambda(c)x + \phi(\Lambda(c)z) - \nabla\phi(\Lambda(c)z)\Lambda(c)z \\
+ \Big(\nabla\phi(\Lambda(c)z)\Lambda(c) -\nabla\phi(\Lambda(c)x)\Lambda(c)\Big) 
(y,y^2, \ldots, y^k)^\top,
\end{gather*}
whereas the right-hand side is
\begin{gather*}
- c^b\phi(x) + c^b\nabla\phi(x)x + c^b\phi(z) - c^b\nabla\phi(z)z 
+ c^b\Big(\nabla\phi(z) -\nabla\phi(x)\Big) 
(y,y^2, \ldots, y^k)^\top.
\end{gather*}
Both terms are polynomials in $y$ of degree $k$, which leads to the identity
\[
\nabla\phi(\Lambda(c)z)\Lambda(c) -\nabla\phi(\Lambda(c)x)\Lambda(c) = 
c^b\Big(\nabla\phi(z) -\nabla\phi(x)\Big).
\]
This is exactly condition \eqref{eq:cond mixed}.
\end{proof}

Recall that the scoring functions of the form at \eqref{eq:score} are essentially all consistent scoring functions for the vector of different moments \cite[Proposition 4.4]{FisslerZiegel2016}.
Using Proposition \ref{prop:mixed moments} it is straight forward to derive consistent scoring functions for (mean, variance) with mixed positively homogeneous score differences.

\begin{cor}\label{cor:EVarHOM}
Let $\F$ be a class of distributions on $\R$ with finite second moments such that the functional $T = (\textup{mean, variance})\colon \F\to\A\subseteq \R\times[0,\infty)$ is surjective, where for all $(x_1,x_2)\in \A$ and $c>0$, $(cx_1,c^2x_2)\in\A$. 
Let Assumptions (F1) and (V1) from \cite{FisslerZiegel2016} be satisfied with the strict $\F$-identification function $V\colon\A\times \R\to\R^2$, $V(x_1,x_2,y) = \big(x_1 - y, x_2 + x_1^2 - y^2 \big)^\top$. 
Let $S\colon\A\times \R\to\R$ be a strictly $\F$-consistent scoring function for $T$ that is (jointly) continuous and for any $y\in\R$, the function $\A\ni x\mapsto S(x,y)$ be twice continuously differentiable.
Then $S$ has mixed positively homogeneous score differences of degree $b\in\R$ if and only if 
\be{eq:form score}
S(x_1,x_2,y) = -\phi(x_1,x_2+x_1^2) + \nabla \phi(x_1,x_2+x_1^2)
\begin{pmatrix}
x_1 - y \\
x_2+x_1^2 - y^2
\end{pmatrix} + a(y),
\ee
where $\phi\colon\A\to\R$ is strictly convex, twice continuously differentiable, and moreover for all $c>0$ the map
\be{eq:condition 2}
\A\ni (x_1, x_2)\mapsto \nabla\phi(cx_1, c^2x_2+c^2x_1^2)\begin{pmatrix}
c & 0 \\ 0 & c^2
\end{pmatrix} - c^b\nabla\phi(x_1,x_2+x_1^2)
\ee
is constant.
\end{cor}
\begin{proof}
The form at \eqref{eq:form score} follows as in the proof of Proposition \ref{prop:mean variance os}. The rest follows by Proposition \ref{prop:mixed moments}.
\end{proof}

It appears that the class of (strictly) convex functions $\phi$ satisfying \eqref{eq:cond mixed} is rather flexible. One subclass is the class of additively separable functions $\phi$. That is, 
\be{eq:additive}
\phi(x) = \sum_{m=1}^k \phi_m(x_m),
\ee
where each $\phi_m$ needs to be convex and $x_m\mapsto c^m\phi_m'(c^mx_m) - c^b\phi_m'(x_m)$ constant. 
Reviewing \citet[Theorem 5]{NoldeZiegel2017} and restricting attention to the case $\A \subseteq (0,\infty)^k$, $\phi_m$ can be an element of the class $\Psi_{b/m}$, where $\Psi_b$ consists of functions $\psi_b\colon(0,\infty)\to\R$ of the form
\[
\psi_b(y) = \begin{cases}
d_0+ d_1y^b/(b(b-1)), & \text{for } b\in\R\setminus\{0,1\} \\
d_0 +d_1y\log(y) + d_2y, & \text{for } b=1 \\
d_0 - d_1\log(y) + d_2y, & \text{for } b=0
\end{cases}
\]
with constants $d_1>0$, $d_0,d_2\in\R$. 
On the other hand, there are choices of $\phi$ not satisfying such a additive decomposition as in \eqref{eq:additive}. One such example can be found in Example \ref{exmp:os for mean, variance} for $b=-2$, and is of the form
$\phi(x_1,x_2) = (x_2 - x_1^2)^{-1}$ for $x_2>x_1^2$.

\section*{Acknowledgement}
Tobias Fissler would like to thank
Tilmann Gneiting, Werner Ehm, Sebastian Lerch, Alexander Jordan, Fabian Kr\"uger, and Jonas Brehmer for stimulating and insightful discussions during a research visit at the Heidelberg Institute for Theoretical Studies.
He is grateful to the Department of Mathematics at Imperial College London for funding his Chapman Fellowship.

\bibliographystyle{plainnat}

\appendix
\section*{Appendix}

\section{Refinements of \cite{FisslerZiegel2016}}\label{appendix:A}

As detailed \cite{Brehmer2017} there are two technicalities that need to be resolved in \citet[Proposition 3.4]{FisslerZiegel2016}: Firstly, due to the particular choice of the integration path in the original version of \citet[Proposition 3.4]{FisslerZiegel2016}, the image of the integration path is not necessarily contained in $\interior(\A)$. Secondly, one needs to assume that the identification function $V$ is locally bounded \emph{jointly} in the two components. Proposition \ref{prop:Proposition 3.4} gives a refined version of \citet[Proposition 3.4]{FisslerZiegel2016}.

\begin{prop}\label{prop:Proposition 3.4}
Assume that $\interior(\A)\subseteq \R^k$ is simply connected and let $T\colon \F\to\A$ be a surjective, elicitable and identifiable functional with a strict $\F$-identification function $V\colon\A\times \O\to\R^k$ and a strictly $\F$-consistent scoring function $S\colon\A\times \O\to\R$. Suppose that Assumption (V1), (V2), (S1) from \cite{FisslerZiegel2016} are satisfied. Let $h$ be the matrix-valued function appearing at \citet[Equation (3.2)]{FisslerZiegel2016}. 

For any $F\in\F$ and any points $x,z\in\interior(\A)$ such that $\gamma\colon[0,1]\to\interior(\A)$ is an integration path with $\gamma(0)=x$, $\gamma(1)=z$ the score difference is necessarily of the form
\begin{align} \label{eq:Prop 3.4}
\bar S(x,F) -\bar S(z,F) = \int_{\gamma} \dint \bar S(\cdot, F) = \int_0^1 h\big(\gamma(\lambda)\big) \bar V\big(\gamma(\lambda),F\big)\gamma'(\lambda)\dint \lambda\,.
\end{align}
Moreover, if Assumptions (F1) and (VS1) from \cite{FisslerZiegel2016} are satisfied and $V$ is locally bounded, then there is a Lebesgue null set $N\subseteq \A\times \O$ such that for all $(x,y)\in N^c$, $(z,y)\in N^c$ it necessarily holds that 
\begin{align} \label{eq:Prop 3.4b}
S(x,y) - S(z,y) = \int_{\gamma} \dint  S(\cdot, y) = \int_0^1 h\big(\gamma(\lambda)\big)  V\big(\gamma(\lambda),y\big)\gamma'(\lambda)\dint \lambda\,,
\end{align}
where again $\gamma\colon[0,1]\to\interior(\A)$ is an integration path with $\gamma(0)=x$, $\gamma(1)=z$.
\end{prop}

\begin{proof}
Equation \eqref{eq:Prop 3.4} follows from \citet[Theorem 3.2]{FisslerZiegel2016} and \citet[Satz 2, p.\ 183]{Koenigsberger2004}. The proof of \eqref{eq:Prop 3.4b} follows the lines of the original proof in \cite{FisslerZiegelSupplementary2016}; cf.\ \citet[Theorem 1.31]{Brehmer2017} for details.
\end{proof}

\end{document}